\documentclass[onefignum,onetabnum]{siamart220329}

\usepackage{lipsum}
\usepackage{amsfonts}
\usepackage{graphicx}
\usepackage{epstopdf}
\usepackage{algorithmic}
\ifpdf
  \DeclareGraphicsExtensions{.eps,.pdf,.png,.jpg}
\else
  \DeclareGraphicsExtensions{.eps}
\fi

\usepackage{amsmath}
\usepackage{amssymb}
\usepackage{hyperref}

\usepackage[shortlabels]{enumitem}
\usepackage{resizegather}
\usepackage{placeins}

\DeclareMathOperator*{\argmin}{\arg\min}

\usepackage{soul}
\usepackage{hyperref}

\usepackage[ruled,linesnumbered,vlined,algo2e]{algorithm2e}

\newsiamremark{remark}{Remark}
\newsiamremark{hypothesis}{Hypothesis}
\newsiamremark{assumption}{Assumption}
\crefname{assumption}{Assumption}{Assumptions}
\newsiamthm{claim}{Claim}

\headers{Front-Descent Methods with Convergence Guarantees}{M. Lapucci, P. Mansueto, D. Pucci}

\title{Effective Front-Descent Algorithms with Convergence Guarantees\thanks{Submitted to the editors DATE.
\funding{The authors declare that no funds, grants, or other support were received during the preparation of this manuscript.}}}

\author{Matteo Lapucci\thanks{Department of Information Engineering, University of Florence, Florence, Italy (\email{matteo.lapucci@unifi.it}, \email{pierluigi.mansueto@unifi.it}, \email{davide.pucci@unifi.it}).}
\and Pierluigi Mansueto\footnotemark[2] \and Davide Pucci \footnotemark[2]}

\usepackage{amsopn}

\ifpdf
\hypersetup{
  pdftitle={Effective Front-Descent Algorithms with Convergence Guarantees},
  pdfauthor={M. Lapucci, P. Mansueto, D. Pucci}
}
\fi

\usepackage{subfig}

\newcommand{\rev}[1]{{#1}}

\begin{document}

\maketitle

\begin{abstract}
In this manuscript, we address continuous unconstrained multi-objective optimization problems and we discuss descent type methods for the reconstruction of the Pareto set. Specifically, we analyze the class of Front Descent methods, which generalizes the Front Steepest Descent algorithm allowing the employment of suitable, effective search directions (e.g., Newton, Quasi-Newton, Barzilai-Borwein).
We provide a deep characterization of the behavior and the mechanisms of the algorithmic framework, and we prove that, under reasonable assumptions, standard convergence results and some complexity bounds hold for the generalized approach. Moreover, we prove that popular search directions can indeed be soundly used within the framework. Then, we provide a completely novel type of convergence results, concerning the sequence of sets produced by the procedure. In particular, iterate sets are shown to asymptotically approach stationarity for all of their points; the convergence result is accompanied by a worst-case iteration complexity bound; additionally, in finite precision settings, the sets are shown to only be enriched through exploration steps in later iterations, and suitable stopping conditions can be devised.
Finally, the results from a large experimental benchmark show that the proposed class of approaches far outperforms state-of-the-art methodologies.
\end{abstract}



\section{Introduction}
In this paper, we are interested in \textit{multi-objective optimization} problems of the form
\begin{equation}
	\label{eq:mo_prob}
	\min_{x\in\mathbb{R}^n}\;F(x) = (f_1(x),\ldots,f_m(x))^\top,
\end{equation}
where $F:\mathbb{R}^n\to\mathbb{R}^m$ is a continuously differentiable function.
Plenty of real-world problems can be formalized as instances of this mathematical framework, see, e.g., \cite{handl2007multiobjective,jin2008pareto,liu2020review,rosso2020multi}; basically, this is the case of all tasks where the quality of a solution is measured based on more than a single criterion - i.e., many objective functions. Clearly, a solution which is simultaneously optimal for each individual objective function is unlikely to exist. For this reason, Pareto's theory is commonly brought into play to characterize the concept of optimality in this class of problems. For a thorough discussion on continuous multi-objective optimization, we refer the reader to \cite{eichfelder2021twenty}.

Pareto optimal solutions provide, in rough terms, unimprovable trade-offs among the objectives. The set of all Pareto optimal solutions is called the \textit{Pareto set} of the problem, and the image of this set through $F$ is referred to as \textit{Pareto front}. In applications, it is often preferable to leave to a knowledgeable decision maker the choice of the most suitable among the optimal trade-offs. In this perspective, we can understand the increasing interest towards algorithms capable of generating (a uniform and spread approximation of) the full Pareto front of multi-objective problems.

Standard approaches to tackle problems of the form \cref{eq:mo_prob} include scalarization and evolutionary algorithms \cite{deb2002fast}. However, both classes of methods have shortcomings: evolutionary algorithms, although flexible, have no theoretical guarantees and poor scalability \cite[Sec.\ 4.2]{lapucci2023memetic}; on the other hand, scalarization requires a careful choice of the weights \cite[Sec.\ 7]{fliege2009newton}, and without suitable regularity assumptions large parts of the front are often even unobtainable \cite[Sec.\ 4.1]{eichfelder2021twenty}.

A more recent family of approaches thus started drawing interest: \textit{descent methods}. Inspired by the seminal work of Fliege and Svaiter \cite{fliege2000steepest}, extensions of classical descent algorithms for single-objective optimization started being proposed \cite{fliege2009newton,gonccalves2022globally,lapucci2023limited,lucambio2018nonlinear,prudente2022quasi}. Initially, these algorithms were designed to produce a single efficient solution of the problem; steepest-descent type methodologies could therefore be employed to refine given solutions \cite{lapucci2023memetic}, or used in a multi-start fashion to produce Pareto front approximations. However, multi-starting single point algorithms does not allow in general to generate spread and uniform fronts. To overcome this limitation, in recent years descent methods have been proposed that explicitly handle a list of candidate solutions: the set of points is improved in a structured way, as a whole, at each iteration \cite{cocchi2020convergence,custodio2011direct,lapucci2023improved,liuzzi2016derivative,mohammadi2024trust}.

In this work, we introduce and discuss a broad class of such algorithms, which we call \textit{Front Descent (FD) methods}. Building upon and generalizing the simpler Front Steepest Descent (FSD) method \cite{cocchi2020convergence,lapucci2023improved}, the proposed algorithmic framework carries out, from each point in the current set of solutions, line searches along directions that are of descent for all or for a subset of objectives. This strategy ensures not only to improve the solutions, driving them towards stationarity, but also to move towards unexplored regions of the Pareto front.

The present work has multiple layers of contributions: first, we provide an improved theoretical characterization of the mechanisms of the base FSD method and, in particular, of the produced sequences of solutions. The new theoretical perspective opens up the definition of the much more general Front Descent framework, enabling the exploitation of various types of descent directions that have widely been shown to be effective in the single-point setting. For instance, the class of \textit{Newton-type directions} is particularly appealing \cite{fliege2009newton,gonccalves2022globally,lapucci2023limited,prudente2022quasi}. 

The general family of algorithms is deeply analyzed from the theoretical point of view: on the one hand, a simple condition on search directions allows to recover the same global convergence guarantees known in the literature for base FSD; moreover, some complexity bound is provided and, for specific choices of directions, we can also prove faster local convergence rates for some sequences of solutions.

On the other hand, under an additional, reasonable assumption on algorithm's implementation, we are able to obtain convergence properties related to any sequence of solutions and, most importantly, to the entire set of points iteratively updated by the procedure. This type of results is, to the best of our knowledge, completely novel in the literature, addressing a major open challenge in the field. The convergence result for the sequence of set of solutions is accompanied by a worst-case iteration complexity bound. All these theoretical findings further allow to provide a rigorous characterization of empirically observed behaviors of FD algorithms and also to properly define suitable stopping conditions for the method.

Finally, thorough numerical experiments show that the proposed methods can exhibit significantly superior performance than base FSD and other state-of-the-art methods for Pareto front reconstruction in continuous multi-objective optimization.

The rest of the paper is organized as follows. In \cref{sec:prelims}, we review the preliminary material needed to read the work; in particular, we recall the basic concepts of multi-objective optimization in \cref{sec:basics}, the front steepest descent algorithm in \cref{sec:fsd} and popular search directions for multi-objective optimization in \cref{sec:search_directions}. Then, we provide a discussion on linked sequences produced by FSD in \cref{sec:linkseq}. In \cref{sec:fda}, we describe the main algorithmic framework proposed in this work. Then, we provide a thorough theoretical analysis of the methodology in \cref{sec:convergence}: in \cref{sec:conv_linked} we give results concerning linked sequences produced by the algorithm, whereas in \cref{sec:conv_front} we carry out the novel analysis concerning the sequence of sets of solutions. Next, in \cref{sec::computational_experiments}, we show the results of the computational experiments assessing the effectiveness of the proposed class of methods. We finally give some concluding remarks in \cref{sec:conclusions}.
\section{Preliminaries}
\label{sec:prelims}
In the next subsections, we provide brief overviews on basic concepts in multi-objective optimization, the front steepest descent method and descent directions.

\subsection{Basic concepts in multi-objective optimization}
\label{sec:basics}
Descent algorithms in multi-objective optimization (MOO) are designed to solve instances of problem \cref{eq:mo_prob} according to the concepts of Pareto optimality. It is therefore useful to begin the discussion recalling some key definitions; first of all, we employ the standard \textit{partial ordering} on $\mathbb{R}^m$: $u\le v$ means $u_i\le v_i$ for all $i=1,\ldots,m$; similarly, $u< v$ if $u_i< v_i$ for all $i$; moreover we denote by $u \lneqq v$ when $u\le v$ and $u\neq v$.
This third relation allows to introduce the notion of \textit{dominance} between solutions of \cref{eq:mo_prob}: given $x, y\in\mathbb{R}^n$, we say that $y$ is \textit{dominated} by $x$ if $F(x)\lneqq F(y)$.

It is now possible to introduce the formal definitions of Pareto optimality.
\begin{definition}
	A point $\bar{x}\in\mathbb{R}^n$ is called 
	\begin{enumerate}[(a)]
		\item \textit{Pareto optimal} (or \textit{efficient}) for problem \cref{eq:mo_prob} if  $\nexists\;y\in\mathbb{R}^n$ s.t.\ $F(y)\lneqq F(\bar{x})$, i.e., $y$ dominates $\bar{x}$;
		\item \textit{weakly Pareto optimal} for problem \cref{eq:mo_prob} if  $\nexists\;y\in\mathbb{R}^n$ s.t.\ $F(y)< F(\bar{x})$.
	\end{enumerate}
\end{definition}
Clearly, Pareto optimality, which is a quite strong condition, implies weak Pareto optimality. Note that local versions of the above optimality concepts can also be given; in that case, the property has to hold in a neighborhood of $\bar{x}$. Under differentiability assumptions, a necessary condition for weak efficiency is \textit{Pareto-stationarity}.
\begin{lemma}
	Let $\bar{x}$ be a (local) weakly Pareto optimal point for problem \cref{eq:mo_prob}. Then, $\bar{x}$ is Pareto-stationary for problem \cref{eq:mo_prob}, i.e.,
	$\min\limits_{d\in\mathbb{R}^n} \max\limits_{j=1,\ldots,m}\nabla f_j(\bar{x})^\top d = 0.$
\end{lemma}
The three conditions - Pareto optimality, weak Pareto optimality and Pareto stationarity - become equivalent under strict convexity assumptions on $F$ \cite{fliege2009newton}.

Taking inspiration from traditional nonlinear programming, a key tool exploited to characterize continuous MOO problems and to define algorithms is that of (multi-objective) descent direction. \rev{Before stating the next definition, we further introduce the set relation $\subsetneq$: given two set $A$ and $B$, $A \subsetneq B$ if and only if $A \subset B$ and $A \ne B$.}

\begin{definition}
	Let $x\in\mathbb{R}^n$. A direction $d\in \mathbb{R}^n$ is referred to as a 
	\begin{itemize}
		\item \textit{common descent direction} at $x$ if, for all $j=1,\ldots,m$, $d$ is a descent direction for $f_j$ at $x$.
		\item \textit{partial descent direction} at $x$ if there exists $I\subsetneq\{1,\ldots,m\}$ such that, for all $j\in I$, $d$ is a descent direction for $f_j$ at $x$.
	\end{itemize}
\end{definition}
\begin{lemma}
	Let $x\in\mathbb{R}^n$ and $d\in\mathbb{R}^n$.
	\begin{enumerate}[(a)]
		\item If $\max\limits_{j=1,\ldots,m} \nabla f_j(x)^\top d < 0,$ then $d$ is a \textit{common descent direction} for $F$ at $x$. 
		\item If there exists $I\subsetneq \{1,\ldots,m\}$ such that $\max\limits_{j\in I} \nabla f_j(x)^\top d < 0$, then  $d$ is a \textit{partial descent direction} for $F$ at $x$.
	\end{enumerate}
\end{lemma}
If we denote by $J_F(x)$ the Jacobian of $F$, the above conditions can be rewritten in a compact way as $J_F(x)d<0$ and $(J_F(x)d)_I<0$, respectively. We can also denote by $\mathcal{D}(x,d)$ the quantity $\max_{j=1,\ldots,m}\nabla f_j(x)^\top d$, so that, letting $\boldsymbol{1}$ be the vector of all ones in $\mathbb{R}^m$, $J_F(x)d\le\boldsymbol{1}\mathcal{D}(x,d)$ and $\mathcal{D}(x,d)<0$ for a common descent direction.

It is easy to see that, at a Pareto-stationary point, there is no common descent direction; however, partial descent directions still might exist in that case. 
We shall now recall the definition of both common and partial \textit{steepest descent directions}.
\begin{definition}
	Let $x\in\mathbb{R}^n$. We have that
\end{definition}
\begin{itemize}
	\item the \textit{steepest common descent direction} \cite{fliege2000steepest} for $F$ at $x$ is defined by: 
	\begin{equation*}
		\label{eq::steepest_common}
		v(x) = \argmin\limits_{d\in\mathbb{R}^n}\max\limits_{j=1,\ldots,m} \nabla f_j(x)^\top d + \frac{1}{2}\|d\|^2;
	\end{equation*}
	\item the \textit{steepest partial descent direction} \cite{cocchi2021pareto,cocchi2020convergence} for $F$ at $x$ w.r.t.\ the subset of objectives with indices in $I\subsetneq \{1,\ldots,m\}$ is defined by:
	\begin{equation*}
		\label{eq::steepest_partial}
		v^I(x) = \argmin\limits_{d\in\mathbb{R}^n}\;\max\limits_{j\in I}\;\nabla f_j(x)^\top d + \frac{1}{2}\|d\|^2.
	\end{equation*}
\end{itemize}
Note that we are allowed to use the equality sign in the above definitions thanks to the uniqueness of the solution for the problems (whose objective function is strongly convex and continuous). 

The optimal value of the steepest descent direction problems can be seen as a  measure for (approximate) stationarity of solutions. We are thus interested in also defining the functions $\theta:\mathbb{R}^n\to\mathbb{R}$ and $\theta^I:\mathbb{R}^n\to\mathbb{R}$ as
$$\theta(x) = \min\limits_{d\in\mathbb{R}^n}\max\limits_{j=1,\ldots,m} \nabla f_j(x)^\top d + \frac{1}{2}\|d\|^2,\qquad \theta^I(x)=\min\limits_{d\in\mathbb{R}^n}\;\max\limits_{j\in I}\;\nabla f_j(x)^\top d + \frac{1}{2}\|d\|^2.$$
Of course, $v^I$ and $\theta^I$ coincide with $v$ and $\theta$ if $I=\{1,\ldots,m\}$. A fundamental property is that mappings $v^I(x)$ and $\theta^I(x)$ are continuous \cite{fliege2000steepest}. Moreover, a point $\bar{x}$ is Pareto-stationary if and only if $\theta(\bar{x})=0$ and $v(\bar{x})= 0$. Also, note that the value of $\theta(x)$ is always nonpositive and that $\mathcal{D}(x,v(x))= -\|v(x)\|^2$ (see, e.g., \cite[Eq.\ (17)]{gonccalves2022globally}).

\subsection{Front Steepest Descent}
\label{sec:fsd}
In this section we describe the prototypical \textit{Front Descent} type algorithm: the \textit{Front Steepest Descent (FSD)} \cite{cocchi2020convergence}. The FSD has been designed to construct an approximation of the Pareto front, which is the set of all Pareto efficient points; obviously, in the nonconvex case, this goal is not reasonably obtainable, similarly as it is not reasonable to ask for certified global optima in the single-objective case. Thus, the actual aim will be that of constructing a front of mutually non-dominated Pareto-stationary points as spread and uniform as possible.

The conceptual scheme of the method can be seen in \cref{alg::improved_FSD}. The pseudocode refers to the perfected version of the algorithm discussed in \cite{lapucci2023improved}, with a couple of minor modifications that will be addressed later in this section. 

\begin{algorithm}[htbp]
	\caption{\texttt{FrontSteepestDescent}} 
	\label{alg::improved_FSD}
	\begin{algorithmic}[1]
		\STATE{Input: $F:\mathbb{R}^n \rightarrow \mathbb{R}^m$, $X^0$ set of mutually nondominated points w.r.t.\ $F$, $\alpha_0>0,$ $\delta\in(0,1),\gamma\in(0,1)$.}
		\STATE{$k = 0$}
		\WHILE{a stopping criterion is not satisfied}
			\STATE{$\hat{X}^k = X^k$}
			\FORALL{$x_c\in X^k$}
				\IF{$x_c \in \hat{X}^k$ \label{step:x_c}}
					\IF{$\theta(x_c)<0$ \label{step.if_cond}}
						\STATE{$\alpha_c^k = \max\limits_{h\in\mathbb{N}} \{\alpha_0\delta^h\mid F(x_c+\alpha_0\delta^h v(x_c))\le F(x_c)+\boldsymbol{1}\gamma\alpha_0\delta^h\mathcal{D}(x_c, v(x_c))\}$\label{step:line_search}}
					\ELSE
						\STATE{$\alpha_c^k=0$ \label{step:end_if_else}}
					\ENDIF
					\STATE{$z_c^k = x_c+\alpha_c^kv(x_c)$ \label{step:z}}
					\STATE{$\hat{X}^k = \hat{X}^k \setminus \left\{y \in \hat{X}^k \mid F(z^k_c) \lneqq F(y)\right\} \cup \left\{z^k_c\right\}$ \label{step:add_zk}}
					\FORALL{$I\subsetneq\{1,\ldots,m\}$ s.t.\ $\theta^I(z_c^k) < 0$ \label{step:inner_for}}
						\IF{$z_c^k\in\hat{X}^k$}
							\STATE{$\alpha_c^I$ = $\max\limits_{h\in\mathbb{N}} \{\alpha_0\delta^h\mid \forall y\in\hat{X}^k,\;\exists\, j \text{ s.t. } f_j(z_c^k+\alpha_0\delta^h v^I(z_c^k)) < f_j(y)\}$ \label{step:second_ls}}
							\STATE{$\hat{X}^k = \hat{X}^k \setminus \left\{y \in \hat{X}^k \mid F\left(z_c^k + \alpha_c^I v^I(z_c^k)\right) \lneqq F(y)\right\} \cup \left\{z_c^k + \alpha_c^I v^I(z_c^k)\right\} $ \label{step:insert_partial}}
						\ENDIF
					\ENDFOR
				\ENDIF	
			\ENDFOR
			\STATE{$X^{k + 1} = \hat{X}^k$}
			\STATE{$k = k + 1$}
		\ENDWHILE
		\RETURN $X^k$
	\end{algorithmic}
\end{algorithm}

As we can infer by its name, the approach makes use of steepest descent directions; the algorithm handles a list of solutions $X^k$, which is updated, at each iteration, as follows.
\begin{itemize}[-]
	\item Sequentially, each point $x_c$ in the current set $X^k$ undergoes a \textit{two-phase} optimization procedure.  
	\item The first phase (steps \ref{step:x_c}-\ref{step:add_zk}), that can be interpreted as a \textit{refinement step}, consists in a single-point steepest descent step starting at $x_c$; we obtain a new point $z_c^k$ that provides a sufficient decrease of all the objective functions, unless $x_c$ itself is already Pareto-stationary; in the latter case, $z_c^k$ is set equal to $x_c$. 
	\item The second phase (steps \ref{step:inner_for}-\ref{step:insert_partial}) aims at \textit{exploring} the objectives space and possibly enriching the set $X^k$ with new nondominated points; in this part, line searches are carried out starting at $z^k_c$ along steepest partial descent directions, considering subsets of objectives such that $\theta^I(z_c^k)<0$. The line search for this phase looks at the entire (updated) set of points $\hat{X}^k$: the stepsize is accepted as soon as it leads to a point which is not dominated by any other point in the current list of solutions. 
	\item Each time a new point is added to the list, a filtering procedure is carried out to remove all points that become dominated. 
\end{itemize}

We notice that, differently than \cite{lapucci2023improved}, at step \ref{step:inner_for} we only consider proper subsets of $\{1,\ldots,m\}$; then, as a second modification, the sufficient decrease condition in the line search at step \ref{step:line_search} employs quantity $\mathcal{D}(x_c,v(x_c))$ instead of $\theta(x_c)$. 
By very simple patches to the proofs, which we do not report here for the sake of brevity, it can be easily seen that these changes
do not alter any theoretical property known for \cref{alg::improved_FSD}: all steps are well defined and the overall procedure enjoys asymptotic convergence properties. \rev{We shall also note that considering all possible subsets of objectives at steps \ref{step:inner_for}-\ref{step:insert_partial} might turn out expensive in practice when $m$ gets high; still, the convergence results would again not be altered if only some subsets were considered, as they are mainly driven by the refinement steps.} Here below, we summarize these theoretical results.

\begin{lemma}[{\cite[ Prop.\ 3.1]{lapucci2023improved}}]
	\label{lemma:ls}
	Instruction \ref{step:line_search} of \cref{alg::improved_FSD} is well defined.
\end{lemma}
\begin{lemma}[{\cite[Prop.\ 3.2-3.3]{lapucci2023improved}}]
	\label{lemma:prelims}
	Let $X^k$ be a set of mutually nondominated points. Then:
	\begin{enumerate}[(a)]
		\item for the entire iteration $k$, $\hat{X}^k$ is a set of mutually nondominated points;
		\item the line search at step \ref{step:second_ls} of \cref{alg::improved_FSD} is always well-defined;
		\item $X^{k+1}$ is a set of mutually nondominated solutions.
	\end{enumerate}
\end{lemma}
The above result\rev{s} thus mean that the algorithm is sound, as long as the initial set $X^0$ contains nondominated solutions.
Another useful property can be stated.
\begin{lemma}[ {\cite[Lemma 3.1]{lapucci2023improved}}]
	\label{lemma:lemma}
	After step \ref{step:add_zk} of \cref{alg::improved_FSD}, $z^k_c$ belongs to $\hat{X}^k$. Moreover, for all $\tilde{k}>k$, there exists $y\in X^{\tilde{k}}$ such that $F(y)\le F(z_c^k)$.
\end{lemma}
Next, in order to characterize the convergence properties of the algorithm, we need to resort to the concept of linked sequence \cite{liuzzi2016derivative}.
\begin{definition}
	Let $\left\{X^k\right\}$ be the sequence of sets of nondominated points produced
	by \cref{alg::improved_FSD}. We define a \textit{linked sequence} as a sequence $\left\{x_{j_k}\right\}$ such that, for any $k=1,2,\ldots$, the point $x_{j_k}\in X^k$ is generated at iteration $k-1$ of \cref{alg::improved_FSD} while processing the point $x_{j_{k-1}}\in X^{k-1}$.
\end{definition}
We are now able to recall the main convergence result.
\begin{proposition}[{\cite[Prop.\ 3.4]{lapucci2023improved}}]
	Let $X^0$ be a set of mutually nondominated points and $x_0\in X^0$ be a point s.t. the set $\mathcal{L}(x_0) = \bigcup_{j=1}^{m}\left\{x\in\mathbb{R}^n\mid f_j(x)\le f_j(x_0)\right\}$ is compact.
	Let $\left\{X^k\right\}$ be the sequence of sets of nondominated points produced by \cref{alg::improved_FSD}. Let $\left\{x_{j_k}\right\}$ be a linked sequence, then it admits accumulation points and every accumulation point is Pareto-stationary for problem \cref{eq:mo_prob}.
\end{proposition} 

\subsection{Search directions} \label{sec:search_directions}
Taking inspiration from the scalar optimization case, several types of search directions have been considered to design iterative algorithms for multi-objective optimization.

Of course, the steepest descent direction $v(x)$ is the most straightforward choice; in fact, the steepest descent direction can be seen as a special case of a more general framework: given symmetric positive definite matrices $B_1,\ldots,B_m$, we define
\begin{gather}
	\label{eq:newt-type-dirs}
	\theta_N(x)=\min_{d\in\mathbb{R}^n}\max_{j=1,\ldots,m} \nabla f_j(x)^\top d + \frac{1}{2}d^\top B_j d ,\qquad v_N(x) = \argmin_{d\in\mathbb{R}^n}\max_{j=1,\ldots,m} \nabla f_j(x)^\top d + \frac{1}{2}d^\top B_j d.
\end{gather}
If we set $B_1=\ldots=B_m=I$, we immediately get back $\theta(x)$ and $v(x)$; if, on the other hand, $B_j = \nabla^2f_j(x)$ for all $j=1,\ldots,m$ - assuming $f$ is twice differentiable - we retrieve multi-objective Newton's direction \cite{fliege2009newton}. Note that it can be easily proved that $\theta_N(x) \le \mathcal{D}(x,v_N(x))/2$, regardless of the definition of the matrices $B_1,\ldots,B_m$ (see, e.g., the proof of \cite[Theorem 5]{gonccalves2022globally}). In general, directions of the form $v_N(x)$ are referred to as \textit{Newton-type} directions in MOO \cite{gonccalves2022globally}, including Quasi-Newton \cite{povalej2014quasi,prudente2022quasi} and limited memory QN directions \cite{lapucci2023limited}. 

As an alternative, \textit{conjugate gradient} methods \cite{lucambio2018nonlinear} define the search direction $d_k$ at a given iteration $k$ as $d_k = v(x^k)+\beta_kd_{k-1}$.
Moreover, a Barzilai-Borwein method has recently been proposed \cite{chen2023barzilai}, where the descent direction is obtained solving
\begin{gather}
	\label{eq:bb_direction}
	\theta_a(x)=\min_{d\in\mathbb{R}^n}\max_{j=1,\ldots,m} \frac{\nabla f_j(x)^\top d}{a_j} + \frac{1}{2}\|d\|^2,\qquad v_a(x) = \argmin_{d\in\mathbb{R}^n}\max_{j=1,\ldots,m} \frac{\nabla f_j(x)^\top d}{a_j} + \frac{1}{2}\|d\|^2,
\end{gather}
where scalars $0<a_{\text{min}}\le a_1,\ldots,a_m\le a_{\text{max}}$ are used to rescale the gradients.

In general, search directions shall satisfy some conditions to ensure that the corresponding descent algorithm enjoys convergence guarantees. A rather simple condition to check is that directions are steepest-descent-related \cite{lapucci2024convergence}.

\section{Characterization of Linked Sequences}
\label{sec:linkseq}
In the definition of \cref{alg::improved_FSD} we made a slight change w.r.t.\ \cite{lapucci2023improved}, which is not substantial in practice and does not spoil the convergence properties either. 
The modification concerns the exploration phase (steps \ref{step:inner_for}-\ref{step:insert_partial}), where the improper subset $I=\{1,\ldots,m\}$ is not employed.

This apparently irrelevant change actually allows us to easily get an enhanced characterization of the linked sequences of points generated by the FSD method. In fact, we will classify linked sequences as either \textit{refining linked sequences} or \textit{exploring linked sequences}, according to the following definition.

\begin{definition}
	Let $\left\{X^k\right\}$ be the sequence of sets of nondominated points produced
	by \cref{alg::improved_FSD} and let  $\left\{x_{j_k}\right\}$ be a linked sequence. If there exists $\bar{k}$ such that, for all $k> \bar{k}$, $x_{j_k} = z_{j_{k-1}}^{k-1}$, i.e., $ x_{j_k} = x_{j_{k-1}}+\alpha_{j_{k-1}}^{k-1} v(x_{j_{k-1}})$, we say that $\left\{x_{j_k}\right\}$ is a \textit{refining linked sequence}.
	
	If, otherwise, for any $\bar{k}$ there exists $k>\bar{k}$ such that $x_{j_k}$ is obtained at iteration $k-1$ by steps \ref{step:inner_for}-\ref{step:insert_partial}, we say that  $\left\{x_{j_k}\right\}$ is an \textit{exploring linked sequence}.
\end{definition}

Basically, a refining linked sequence is such that, from a certain iteration onwards, only steepest common descent steps are carried out so that the current solution is improved and driven towards Pareto-stationarity. On the other hand, exploring sequences include an infinite number of steps of enrichment of the Pareto set approximation. The two different kinds of linked sequences are illustrated in \cref{fig::linked_sequences}. 
\begin{figure}
	\subfloat[Refining linked sequences]{\includegraphics[width=0.45\textwidth]{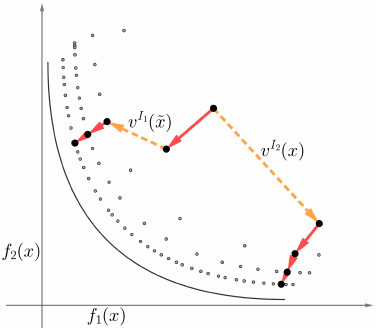}}
	\hfil
	\subfloat[Exploring linked sequences]{\includegraphics[width=0.45\textwidth]{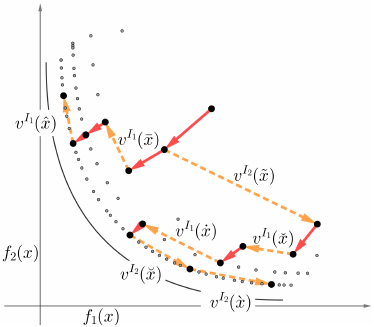}}
	\caption{Examples of linked sequences on a bi-objective problem. 
		The red solid arrows indicate common descent (refining) steps, whereas orange dotted ones denote
		 partial descent (exploring) steps. For the latter, we indicate the considered subset of objective functions, with $I_1 = \{1\}$ and $I_2 = \{2\}$. 
	The black solid line denotes the Pareto front of the problem.}
	\label{fig::linked_sequences}
\end{figure}
Note that the aforementioned modification to \cref{alg::improved_FSD} is fundamental for the above definitions. Indeed, the exploration step w.r.t.\ the subset $I=\{1,\ldots,m\}$ would most often result in $z_c^k$ being removed from $\hat{X}^k$, making the concept of refining sequence somewhat unsubstantial.
Now, the above distinction immediately allows us to give the following complexity result for all refining linked sequences. 
\begin{proposition}
	\label{prop:compl_rls}
	Assume the gradients $\nabla f_1(x),\ldots,\nabla f_m(x)$ are Lipschitz continuous with constants $L_1,\ldots, L_m$ ($L_{\max} = \max_{j=1,\ldots,m}L_j$), and that a nonempty subset $J\subseteq\{1,\ldots,m\}$ exists such that $f_j$ is bounded below for all $j\in J$. 
	Let $\left\{X^k\right\}$ be the sequence of sets of nondominated points produced by \cref{alg::improved_FSD} and let $\left\{x_{j_k}\right\}$ be a refining linked sequence. Assume $\bar{k}$ is the \rev{first iteration index} such that $x_{j_k} = z_{j_{k-1}}^{k-1}$ for all $k>\bar{k}$. Then, for any $\epsilon>0$, after at most $k_\text{max}$ iterations we get $\|v(x_{j_{k_\text{max}}})\|\le \epsilon$, with $k_\text{max} = \bar{k} + k_\text{ref}$ and $k_\text{ref} = \mathcal{O}(\frac{1}{\epsilon^2})$.
\end{proposition} 
\begin{proof}
	The proof straightforwardly follows from the results in \cite{fliege2019complexity,lapucci2024convergence}, noting that the sequence $\{x_{j_k}\}_{k\ge \bar{k}}$ exactly coincides with the sequence produced by the multi-objective steepest descent method with initial point $x_{j_{\bar{k}}}$.
\end{proof}

\begin{remark}
	\label{remark:compl}
	The above result needs to be interpreted carefully; from the one hand, no bound is provided on the iteration $\bar{k}$ where the last exploration step is carried out along the linked sequence. Of course this number might be arbitrarily large and dominate the sum defining $k_\text{max}$. However, \cref{prop:compl_rls} allows to characterize the cost of the tail of purely refining steps of the sequence. In other words, once a point is generated by an exploration step, we are guaranteed that in $\mathcal{O}(\frac{1}{\epsilon^2})$ additional iterations either it will have been driven to $\epsilon$-Pareto-stationarity or it will have been deleted from the list as a result of other points dominating it. 
\end{remark}

The other element of interest we get from the characterization of linked sequences lies in the possibility of soundly boosting the refinement step, making use of some of the efficient descent methods proposed in the literature for single-point MOO. From a computational perspective, such change might allow to drive more quickly solutions to Pareto stationarity and save computational resources to be employed to better spread the Pareto front. 
\section{Front Descent Algorithms}
\label{sec:fda}
In this section, we present the main algorithm introduced in this work. The fundamental idea consists in exploiting better search directions than simple steepest descent in the refinement step of Front Descent, in order to speed up convergence towards Pareto stationarity of refining linked sequences. The analysis in following sections will also highlight that the overall number of refining steps is consequently cut down, so that more computational resources can be invested for the improvement of the spread of the constructed Pareto front.
The general \textit{Front Descent method} is formally described in \cref{alg::F-Newton}. 

\begin{algorithm}[htb]
	\caption{\texttt{FrontDescent}} 
	\label{alg::F-Newton}
	\begin{algorithmic}[1]
		\STATE{Input: $F:\mathbb{R}^n \rightarrow \mathbb{R}^m$, $X^0$ set of mutually nondominated points w.r.t.\ $F$, $\alpha_0>0,$ $\delta\in(0,1),\gamma\in(0,1)$, $\Gamma_1>0,\Gamma_2>0$, $\{\sigma_k\}\subseteq\mathbb{R}_+$.}
		\STATE{$k = 0$}
		\WHILE{a stopping criterion is not satisfied}
			\STATE{$\hat{X}^k = X^k$}
			\FORALL{$x_c\in X^k$ \label{step:main_for}}
				\IF{$x_c \in \hat{X}^k$ \label{step:x_c_N}}
					\IF{$\theta(x_c)<-\sigma_k$ \label{step.if_cond_N}}
						\STATE{Compute a common descent direction $v_D(x_c)$ 
						\label{step:search_dir}}
						\IF{$\mathcal{D}(x_c,v_D(x_c))\le-\Gamma_1\|v(x_c)\|^2$ \textbf{and}  $\|v_D(x_c)\|\le\Gamma_2\|v(x_c)\|$ \label{step:sdr}}
							\STATE{Set $d_c = v_D(x_c)$}
						\ELSE
							\STATE{Set $d_c = v(x_c)$ \label{step:sd_dir}}
						\ENDIF
						\STATE{$\alpha_c^k = \max\limits_{h\in\mathbb{N}} \{\alpha_0\delta^h\mid F(x_c+\alpha_0\delta^h d_c)\le F(x_c)+\mathbf{1}\gamma\alpha_0\delta^h\mathcal{D}(x_c, d_c)\}$\label{step:line_search_N}}
					\ELSE
						\STATE{$d_c = v(x_c)$\\$\alpha_c^k=0$ \label{step:end_if_else_N}}
					\ENDIF
					\STATE{$z_c^k = x_c+\alpha_c^kd_c$ \label{step:z_N}}
					\STATE{$\hat{X}^k = \hat{X}^k \setminus \left\{y \in \hat{X}^k \mid F(z^k_c) \lneqq F(y)\right\} \cup \left\{z^k_c\right\}$ \label{step:add_zk_N}}
					\FORALL{$I\subsetneq\{1,\ldots,m\}$ s.t.\ $\theta^I(z_c^k) < 0$ \label{step:inner_for_N}}
						\IF{$z_c^k\in\hat{X}^k$}
							\STATE{$\alpha_c^I$ = $\max\limits_{h\in\mathbb{N}} \{\alpha_0\delta^h\mid \forall y\in\hat{X}^k,\;\exists\, j \text{ s.t. } f_j(z_c^k+\alpha_0\delta^h v^I(z_c^k)) < f_j(y)\}$ \label{step:second_ls_N}}
							\STATE{$\hat{X}^k = \hat{X}^k \setminus \left\{y \in \hat{X}^k \mid F\left(z_c^k + \alpha_c^I v^I(z_c^k)\right) \lneqq F(y)\right\} \cup \left\{z_c^k + \alpha_c^I v^I(z_c^k)\right\}$ \label{step:insert_partial_N}}
						\ENDIF
					\ENDFOR	
				\ENDIF	
			\ENDFOR \label{step:end_loop}
			\STATE{$X^{k + 1} = \hat{X}^k$}
			\STATE{$k = k + 1$}
		\ENDWHILE
		\RETURN $X^k$
	\end{algorithmic}
\end{algorithm}

\rev{The main structure of the algorithm traces back to that of \cref{alg::improved_FSD}, with each iteration operating on a list of mutually nondominated solutions that sequentially undergo a refinement step and serve as starting points for exploration steps. 

Of course, some crucial elements change with respect to the FSD method.
The most evident difference from \cref{alg::improved_FSD} of course lies in the definition of the search direction for refinement, at steps \ref{step:search_dir}-\ref{step:sd_dir}: any common descent direction can in principle be considered. However, we can actually observe that the desired direction is not necessarily employed at steps \ref{step:line_search_N} and \ref{step:z_N}; indeed, the control at step \ref{step:sdr} checks for the direction $v_D(x_c)$ being steepest-descent-related according to \cite{lapucci2024convergence}. If the condition is not satisfied, we resort to the steepest descent direction $v(x_c)$. 

The presence of the steepest-descent-related safeguard allows to obtain various global convergence and worst case complexity results, as we will detail in \cref{sec:convergence}.
Moreover, it will be possible to state conditions guaranteeing that classes of search directions always pass the control.

The other major difference with respect to the baseline FSD is arguably more subtle, and concerns the Pareto-stationarity check at step \ref{step.if_cond_N}. In fact, we set in \cref{alg::F-Newton} an explicit threshold $\sigma_k$ for the  check at step \ref{step.if_cond_N} in iteration $k$. 
With this expedient, we are able to formally address the computational practice, where nonzero thresholds have to be set to robustly deal with finite precision. 

The computational behavior of the algorithm hence depends on a predefined sequence of thresholds $\{\sigma_k\}\subseteq\mathbb{R}_+$. From the one hand, letting $\sigma_k\to 0$ we still can reach arbitrary precision in the limit. On the other hand, 
in software implementations of the algorithm we might want to set a constant nonzero value $\sigma>0$ for $\sigma_k$, corresponding to the acceptable accuracy for considering a point Pareto-stationary. In this second scenario, as we will see later in this work, computationally nice properties of the procedure can be formally guaranteed. 

It might be important to underline that, instead, the exploration phase, remains unchanged w.r.t.\ \cref{alg::improved_FSD}: the directions used to carry out the exploration steps are the simple (partial) steepest descent ones.
Similarly as what we observed for \cref{alg::improved_FSD}, carrying out exploration steps for all proper subsets of $\{1,\ldots,m\}$ might become expensive for problems where $m$ (and also $n$) is relatively large. However, relaxing this requirement would once again have no jeopardizing effect on the convergence results presented in the following.}

\section{Convergence Analysis}
\label{sec:convergence}
In this section, we provide the theoretical analysis for \cref{alg::F-Newton}. The analysis is split into two parts: the former one concerns standard results on linked sequences, similar to those given in the literature for FSD and similar methods. The latter one concerns a completely novel type of convergence and complexity results for front descent algorithms, allowing to directly characterize the behavior of the sequence of whole sets $\{X^k\}$. 
\subsection{Convergence Analysis of Linked Sequences}
\label{sec:conv_linked}
Before turning to the convergence analysis of \cref{alg::F-Newton}, we need to give some results, similar to \cref{lemma:ls,lemma:prelims,lemma:lemma} for FSD, stating that the procedure is well defined.
\begin{lemma}
	\label{lemma:ls_N}
	The line search at step \ref{step:line_search_N} of \cref{alg::F-Newton} is well defined.
\end{lemma}
\begin{proof}
	We have two cases, depending on the definition of direction $d_c$.
	\begin{enumerate}[(i)]
		\item $d_c=v(x_c)$ - the direction is the steepest descent direction; recalling that $\theta(x_c)<-\sigma_k\le0$, we have $\mathcal{D}(x_c, d_c)<0$ and thus, by \cite[Lemma 4]{fliege2000steepest}, we get the result.
		\item $d_c=v_D(x_c)$ - the direction is the tentative direction; recalling that $\theta(x_c)<0$, i.e., $v(x_c)\neq 0$, we get $\mathcal{D}(x_c,d_c) \le -\Gamma_1\|v(x_c)\|^2<0$ by the first condition of step \ref{step:sdr}. By \cite[Lemma 4]{fliege2000steepest}, we again get the result.
	\end{enumerate}
\end{proof}
\begin{lemma}
	\label{lemma:prelims_N}
	Let $X^k$ be a set of mutually nondominated points. Then:
	\begin{enumerate}[(a)]
		\item throughout the entire iteration $k$, $\hat{X}^k$ is a set of mutually nondominated points;
		\item the line search at step \ref{step:second_ls_N} of \cref{alg::F-Newton} is always well-defined;
		\item $X^{k+1}$ is a set of mutually nondominated solutions.
	\end{enumerate}
\end{lemma}
\begin{proof}
	Identical proofs as for \cite[Prop.\ 3.2-3.3]{lapucci2023improved} hold for this case.
\end{proof}
\begin{lemma}
	\label{lemma:lemma_N}
	After step \ref{step:add_zk_N} of \cref{alg::F-Newton}, $z^k_c$ belongs to $\hat{X}^k$. Moreover, for all $\tilde{k}>k$, there exists $y\in X^{\tilde{k}}$ such that $F(y)\le F(z_c^k)$.
\end{lemma}
\begin{proof}
	The proof of \cite[Lemma 3.1]{lapucci2023improved} holds for this result.
\end{proof}

\noindent We provide an additional lemma that will be useful in the following.
\begin{lemma}
	\label{lemma:compact_set}
	Let $X^0$ be a set of mutually nondominated points and $x_0\in X^0$ be a point such that the set $\mathcal{L}(x_0) = \bigcup_{j=1}^{m}\left\{x\in\mathbb{R}^n\mid f_j(x)\le f_j(x_0)\right\}$ is compact.
	Let $\left\{X^k\right\}$ be the sequence of sets of nondominated points produced by \cref{alg::F-Newton}.
	For all $k=0,1,\ldots,$ and for all $x\in X^k$, we have $x\in\mathcal{L}(x_0)$.
\end{lemma}
\begin{proof}
	Let $k$ be a generic iteration and consider any $x\in X^k$. First, since $X^0$ is a set of nondominated points and recalling point (c) of \cref{lemma:prelims_N}, we know that $X^k$ is also a set of mutually nondominated solutions. We now have one of two cases.
	\begin{itemize}
		\item[-] $x_0\in X^k$; by the nondomination property of $X^k$, there exists at least one index $j$ such that $f_j(x)\le f_j(x_0)$; thus, $x\in \mathcal{L}(x_0)$. 
		\item[-] $x_0\notin X^k$; by \cref{lemma:lemma_N}, we know that there exists a point $y_k\in X^k$ such that $F(y_k)\le F(x_0)$; since $y_k$ does not dominate $x$, there exists $h \in \left\{1,\ldots, m\right\}$ such that $f_h(x)\le f_h(y_k)\le f_h(x_0)$, i.e., $x\in\mathcal{L}(x_0)$ also in this case.
	\end{itemize}
\end{proof}

We are now able to state the first global convergence result, similar to the one available for FSD, concerning all linked sequences produced by the algorithm (both refining and exploring), in case $\sigma_k$ is driven to 0.
\begin{proposition}
	Let $X^0$ be a set of mutually nondominated points and $x_0\in X^0$ be a point such that the set $\mathcal{L}(x_0) = \bigcup_{j=1}^{m}\left\{x\in\mathbb{R}^n\mid f_j(x)\le f_j(x_0)\right\}$ is compact.
	Let $\left\{X^k\right\}$ be the sequence of sets of nondominated points produced by \cref{alg::F-Newton}, with $\sigma_k\to0$. Let $\left\{x_{j_k}\right\}$ be a linked sequence, then it admits accumulation points and every accumulation point is Pareto-stationary for problem \cref{eq:mo_prob}.
\end{proposition}
\begin{proof}
	By \cref{lemma:compact_set},
	the entire sequence $\{x_{j_k}\}$ is contained in the compact set $\mathcal{L}(x_0)$, and thus admits accumulation points. 
	
	We can now take an accumulation point $\bar{x}$ of a linked sequence $\{x_{j_k}\}$, i.e., we consider a subsequence $K\subseteq\{1,2,\ldots\}$ such that
	$x_{j_k} \to \bar{x}$ for $k\to\infty$, $k\in K$.
	Let us assume, by contradiction, that $\bar{x}$ is not Pareto-stationary, i.e., $\|v(\bar{x})\|>0$ and $\theta(\bar{x})<0$. {Recalling that $\sigma_k\to 0$ and that $\theta$ is continuous, we have for $k$ sufficiently large that $\theta(x_{j_k})<-\sigma_k$. Moreover,} 
	by the continuity of $v(\cdot)$, there exists $\varepsilon>0$ such that for all $k\in K$ sufficiently large we have $\|v(x_{j_k})\|>\varepsilon>0$.

	Let $z_{j_k}^k = x_{j_k}+\alpha_{j_k}^kd_{j_k}$ the point obtained at step \ref{step:z_N} of the algorithm when processing $x_{j_k}$; {for $k$ sufficiently large, since $\theta(x_{j_k})<-\sigma_k$, $\alpha_{j_k}$ and $d_{j_k}$ are obtained by instructions \ref{step:search_dir}-\ref{step:line_search_N}.}
	By the continuity of $v(\cdot)$, $v(x_{j_k})\to v(\bar{x})$ for $k\in K$, $k\to\infty$.
	By the second condition in the control at step \ref{step:sdr}, either $d_{j_k} = v(x_{j_k})$ or $\|d_{j_k}\|\le \Gamma_2 \|v(x_{j_k})\|$; therefore, $\|d_{j_k}\|\le \max\{1,\Gamma_2\}\|v(x_{j_k})\|.$ Since $v(x_{j_k})\to v(\bar{x})$, the sequence $d_{j_k}$ is bounded for $k\in K$. Moreover, $\alpha_{j_k}^k\in[0,\alpha_0]$, which is a compact set. Therefore, there exists a further subsequence $K_1\subseteq K$ such that, for $k\in K_1$, $k\to\infty$, we have 
	$$\lim_{\substack{k\to\infty\\k\in K_1}}\alpha_{j_k}^k= \bar{\alpha}\in[0,\alpha_0],\qquad\qquad \lim_{\substack{k\to\infty\\k\in K_1}}d_{j_k}= \bar{d},\qquad\qquad \lim_{\substack{k\to\infty\\k\in K_1}}z_{j_k}^k= \bar{x} + \bar{\alpha}\bar{d} = \bar{z}.$$
	By the definition of $\alpha_{j_k}^k$ and $z_{j_k}^k$ (steps \ref{step:line_search_N}-\ref{step:z_N}) we have that $$F(z_{j_k}^k)\le F(x_{j_k}) + \boldsymbol{1}\gamma \alpha_{j_k}^k \mathcal{D}(x_{j_k},d_{j_k}).$$
	By the first condition at step \ref{step:sdr}, we either have $\mathcal{D}(x_{j_k},d_{j_k})\le -\Gamma_1\|v(x_{j_k})\|^2$ or $\mathcal{D}(x_{j_k},d_{j_k}) = \mathcal{D}(x_{j_k},v(x_{j_k}))= -\|v(x_{j_k})\|^2$; thus, we can write 
	$$F(z_{j_k}^k)\le F(x_{j_k}) - \boldsymbol{1}\min\left\{{1},\Gamma_1\right\}\gamma \alpha_{j_k}^k \|v(x_{j_k})\|^2.$$
	Taking the limits for $k\in K_1$, $k\to \infty$, we get
	\begin{equation}
		\label{eq:main_1}
		F(\bar{z}) \le F(\bar{x}) - \boldsymbol{1}\min\left\{{1},\Gamma_1\right\}\gamma \bar{\alpha} \|v(\bar{x})\|^2\le F(\bar{x})-\boldsymbol{1}\min\left\{{1},\Gamma_1\right\}\gamma \bar{\alpha}\varepsilon^2.
	\end{equation}

	Let us now denote by $k_1(k)$ the smallest index in $K_1$ such that $k_1(k)>k$. For any $k\in K_1$, we know from \cref{lemma:lemma_N} that there exists $y_{j_{k_1(k)}}\in X^{k_1(k)}$ such that $F(y_{j_{k_1(k)}})\le F(z_{j_k}^k)$; since $x_{j_{k_1(k)}}\in X^{k_1(k)}$ too and, by \cref{lemma:prelims_N}, the points in $X^{k_1(k)}$ are mutually nondominated, there must exist $h(k)\in\{1,\ldots,m\}$ such that
	$$f_{h(k)}(x_{j_{k_1(k)}})\le f_{h(k)}(y_{j_{k_1(k)}})\le f_{h(k)}(z_{j_k}^k).$$
	Taking the limits along a further subsequence $K_2\subseteq K_1$ such that $h(k) = h$ for all $k\in K_2$, we get
	$f_h(\bar{x})\le f_h(\bar{z})$.
	Recalling \cref{eq:main_1}, we then obtain
	$$f_h(\bar{x})\le f_h(\bar{z})\le f_h(\bar{x})-\min\left\{{1},\Gamma_1\right\}\gamma\bar{\alpha}\varepsilon^2.$$
	Noting that $\bar{\alpha}\in[0,\alpha_0]$, $\varepsilon>0$ and $\gamma>0$, we deduce that $\bar{\alpha}=\lim_{k\to \infty,k\in K_2}\alpha_{j_k}^k=0$. 
	Being $\bar{x}$ not stationary, for all $k\in K_2$ sufficiently large $\theta(x_{j_k})<0$ holds and $\alpha_{j_k}^k$ is therefore defined at step \ref{step:line_search_N}. 
	Since $\alpha_{j_k}^k\to 0$, given any $q\in\mathbb{N}$, for all $k\in K_2$ large enough we necessarily have $\alpha_{j_k}^k< \alpha_0\delta^q$. This means that the Armijo condition $F\left(x_{j_k}+\alpha d_{j_k}\right)\le F(x_{j_k})+\boldsymbol{1}\gamma \alpha \mathcal{D}(x_{j_k}, d_{j_k})$ is not satisfied by $\alpha= \alpha_0\delta^q$, i.e., for some $\tilde{h}(k)$ we have $$f_{\tilde{h}(k)}\left(x_{j_k}+\alpha_0\delta^q d_{j_k}\right) > f_{\tilde{h}(k)}(x_{j_k})+\gamma \alpha_0\delta^q \mathcal{D}(x_{j_k}, d_{j_k}).$$
	Taking the limits along a suitable subsequence such that $\tilde{h}(k)= \tilde{h}$, we get $$f_{\tilde{h}}\left(\bar{x}+\alpha_0\delta^q\bar{d}\right)\ge f_{\tilde{h}}(\bar{x})+\gamma \alpha_0\delta^q\mathcal{D}(\bar{x}, \bar{d}).$$
	Being $q$ arbitrary and recalling \cite[Lemma 4]{fliege2000steepest}, it must be $\mathcal{D}(\bar{x},\bar{d})\ge 0$. 
	On the other hand, we know that $\mathcal{D}(x_{j_k},d_{j_k})\le -\min\left\{{1},\Gamma_1\right\}\|v(x_{j_k})\|^2$. Taking the limits we get
	$\mathcal{D}(\bar{x},\bar{d})\le -\min\left\{{1},\Gamma_1\right\}\|v(\bar{x})\|^2<0$. 
	The contradiction finally ends the proof.
\end{proof}

Similarly as in \cref{prop:compl_rls} for FSD, we can also get a complexity result for refining linked sequences (cfr.\ \cref{remark:compl}).  
\begin{proposition}
	\label{prop:complexity_linked_fd}
	Assume the gradients $\nabla f_1(x),\ldots,\nabla f_m(x)$ are Lipschitz continuous with constants $L_1,\ldots, L_m$ ($L_{\max} = \max_{j=1,\ldots,m}L_j$), and that a nonempty subset $J\subseteq\{1,\ldots,m\}$ exists such that $f_j$ is bounded below for all $j\in J$. 
	Let $\left\{X^k\right\}$ be the sequence of sets of nondominated points produced by \cref{alg::F-Newton} with $\sigma_k= 0$ for all $k$ and let $\left\{x_{j_k}\right\}$ be a refining linked sequence. Assume $\bar{k}$ is the \rev{first iteration index} such that $x_{j_k} = z_{j_{k-1}}^{k-1}$ for all $k>\bar{k}$. Then, for any $\epsilon>0$, after at most $k_\text{max}$ iterations we get $\|v(x_{j_{k_\text{max}}})\|\le \epsilon$, with $k_\text{max} = \bar{k} + k_\text{ref}$ and $k_\text{ref} = \mathcal{O}(\frac{1}{\epsilon^2})$.
\end{proposition}
\begin{proof}
	In order to prove the result, we shall note that the sequence $\{x_{j_k}\}_{k\ge \bar{k}}$ exactly coincides with the sequence produced by the single-point multi-objective descent method with Armijo-type line search: $x_{j_{k+1}} = x_{j_k} + \alpha_{j_k}^kd_{j_k}$. By \cite{lapucci2024convergence}, in order to prove the result it is sufficient to show that the sequence of descent directions $\{d_{j_k}\}_{k\ge \bar{k}}$ is steepest-descent-related, satisfying, for some $c_1,c_2>0$,
	$\mathcal{D}(x_{j_k},d_{j_k})\le-c_1\|v(x_{j_k})\|^2$ and $\|d_{j_k}\|\le c_2\|v(x_{j_k})\|$.
	
	Now, $d_{j_k}$ is either equal to $v(x_{j_k})$ or $v_D(x_{j_k})$. In the former case, the above conditions straightforwardly hold with $c_1=1$ and $c_2=1$.
	
	On the other hand, if $d_{j_k}=v_D(x_{j_k})$, from the control at step \ref{step:sdr} we are guaranteed that the above conditions hold with $c_1=\Gamma_1$ and $c_2=\Gamma_2$. 
	By setting $c_1=\min\{1,\Gamma_1\}$ and $c_2=\max\{1,\Gamma_2\}$ we thus get the thesis.
\end{proof}

\begin{remark}
	\rev{For the ease of presentation and to align with the result of \cref{prop:compl_rls}, the above proposition is stated for the ideal case of ${\sigma_k}$ being constantly set to zero. If we wanted to address the more general case of $\sigma_k\to 0$, we would have to take into account that the linked sequence could present ``null'' steps, for which $x_{j_k} = x_{j_{k+1}}$ because $\theta(x_{j_k})\ge -\sigma_k$. While the number of total iterations of the linked sequence could
	be arbitrarily large, depending on how $\sigma_k$ converges to zero, we could easily retrieve a complexity bound analogous to the one of \cref{prop:compl_rls} for the ``non-null'' steps, i.e., we could bound the number of iterations where the point in the refining linked sequence is actually processed. Notably, the entire computational cost corresponding to the refining linked sequence is indeed due to the non-null steps, as null steps do not require to carry out any operation.}
\end{remark}

Next, we provide results concerning conditions guaranteeing that Newton-type directions will indeed be steepest-descent-related, i.e., that $v_D(x_c) = v_N(x_c)$ will satisfy the conditions at step \ref{step:sdr} of \cref{alg::F-Newton}.

\begin{proposition}
	\label{prop::vN_sdr}
	Assume $\Gamma_1, \Gamma_2$ are set in \cref{alg::F-Newton} such that $\Gamma_1\le\frac{c_1}{2c_2^2}$ and $\Gamma_2\ge\frac{1}{c_1}$, with $c_2\ge c_1>0$. Further assume that direction $v_D(x_c)$ is a Newton-type direction, computed according to \eqref{eq:newt-type-dirs} with matrices $B_j(x_c)$ such that, for all $j=1,\ldots,m$, $$c_1\le \lambda_\text{min}(B_j(x_c))\le \lambda_\text{max}(B_j(x_c))\le c_2,$$ where $\lambda_\text{min}(B_j(x_c))$ and $\lambda_\text{max}(B_j(x_c))$ denote the minimum and maximum eigenvalues of $B_j(x_c)$ respectively. Then, $v_D(x_c) = v_N(x_c)$ always passes the tests at step \ref{step:sdr}.
\end{proposition}  
\begin{proof}
	By the result in \cite[Lemma 2.8]{lapucci2024convergence}, we have that
	$\|v_N(x_c)\|\le \frac{1}{c_1}\|v(x_c)\|$ and that $\mathcal{D}(x_c,v_N(x_c))\le-\frac{c_1}{2c_2^2}\|v(x_c)\|^2$.
	This proves the assertion.
\end{proof}

The above result ensures us that, if we choose to employ Newton-type directions in the refinement step, and if we employ some safeguarding technique in the definition of matrices $B_j(x_c)$, making them uniformly positive definite, there will actually be no need for the control at step \ref{step:sdr} and to eventually resort to the steepest descent direction.

The situation is somewhat different if we are interested in using the true Hessian matrices $\nabla^2 f_j(x_c)$ to construct $v_N(x_c)$, i.e., if we want to employ pure Newton's direction. In fact, in order to recover interesting convergence speed results for refining linked sequences, we need Newton's direction not to be altered in later iterations.
We prove by the following result that this is actually the case when a refining linked sequence gets close to a second-order Pareto-stationary solution. {Similarly as in \cref{prop:complexity_linked_fd}, we state the result assuming $\{\sigma_k\}$ is constantly zero. In the general case, analogous results hold if we just ignore ``null iterations'' in the linked sequence.}

\begin{proposition}
	\label{prop:ls_newton}
	Assume that $F$ is twice continuously differentiable over $\mathbb{R}^n$. Let $\{x_{j_k}\}$ be a refining linked sequence produced by \cref{alg::F-Newton} with $\sigma_k=0$ for all $k$, $\alpha_0=1$, $\gamma\in(0,\frac{1}{2})$ and $v_D(x_{j_k})=v_N(x_{j_k})$ computed at step \ref{step:search_dir} according to \eqref{eq:newt-type-dirs} with matrices $B_h(x_{j_k}) = \nabla^2 f_h(x_{j_k}) + \eta_h^kI$ for all $h=1,\ldots,m$, where 
	\begin{equation}
		\label{eq::def_eta}
		\eta_h^k = 
		\begin{cases} 
			-\lambda_\text{min}(\nabla^2 f_h(x_{j_k})) + \kappa, & \mbox{if }\lambda_\text{min}(\nabla^2 f_h(x_{j_k})) \le 0, \\ 
			0, & \mbox{otherwise},
		\end{cases}
	\end{equation}
	and $\kappa > 0$. Let $\bar{x}$ be a local Pareto optimal point s.t.\ $\nabla^2 f_h(\bar{x})$ is positive definite for all $h=1,\ldots,m$. Assume that $$\Gamma_1\le\frac{\omega^3}{2}\frac{\min\limits_{h=1,\ldots,m}\lambda_\text{min}(\nabla^2 f_h(\bar{x}))}{\max\limits_{h=1,\ldots,m}\lambda_\text{max}^2(\nabla^2 f_h(\bar{x}))}, \qquad \Gamma_2\ge\frac{1}{\omega\min\limits_{h=1,\ldots,m}\lambda_\text{min}(\nabla^2 f_h(\bar{x}))},$$ with $\omega \in (0, 1)$. Then, there exist a neighborhood $\mathcal{N}(\bar{x})$ of $\bar{x}$ and $0 < \rho \le r$ such that, if $x_{j_{\bar{k}}} \in \mathcal{B}(\bar{x},\rho) \subsetneq \mathcal{N}(\bar{x})$ for some $\bar{k}$ s.t.\ $x_{j_{k+1}} = z_{j_{k}}^k$ for all $k \ge \bar{k}$, the following properties hold for all $k\ge\bar{k}$:
	\begin{enumerate}[(i)]
		\item $\eta_h^k = 0$ for all $h=1,\ldots,m$, i.e., $v_{N}(x_{j_k})$ is the Newton's direction at $x_{j_k}$;
		\item $d_{j_k} = v_N(x_{j_k})$, i.e., $v_N(x_{j_k})$ satisfies the conditions at step \ref{step:sdr};
		\item $\alpha^k_{j_k}=1$, i.e., no backtracking is needed in the Armijo-type line search at step \ref{step:line_search_N};
		\item $x_{j_k}\in\mathcal{B}(x_{j_{\bar{k}}},r) \subsetneq \mathcal{N}(\bar{x})$.
	\end{enumerate}
	Furthermore, the sequence $\left\{x_{j_k}\right\}$ converges superlinearly to a local Pareto optimal point $x^\star \in \mathcal{B}(x_{j_{\bar{k}}},r)$. If $\nabla^2f_h(\cdot)$ are additionally Lipschitz continuous in $\mathcal{N}(\bar{x})$ for all $h=1,\ldots,m$, the convergence is quadratic.
\end{proposition}
\begin{proof}
	By the continuity of the second derivatives and the positive definiteness of $\nabla^2 f_h(\bar{x})$ ($h=1,\ldots,m$), there exists a neighborhood $\mathcal{N}(\bar{x})$ of $\bar{x}$ such that, for all $x \in \mathcal{N}(\bar{x})$ and $h=1,\ldots,m$, 
	\begin{equation}
		\label{eq::la_mi_ma}
		0 < \omega\lambda_\text{min}(\nabla^2 f_h(\bar{x})) \le \lambda_\text{min}(\nabla^2 f_h(x)), \quad \lambda_\text{max}(\nabla^2 f_h(x)) \le \frac{1}{\omega}\lambda_\text{max}(\nabla^2 f_h(\bar{x})).
	\end{equation}
	
	Now, let us consider an iteration $k \ge \bar{k}$ such that $x_{j_k} \in \mathcal{N}(\bar{x})$. By equation \eqref{eq::la_mi_ma}, we get that, for all $h=1,\ldots,m$, $\lambda_\text{min}(\nabla^2 f_h(x_{j_k})) > 0$ and, as a consequence, $\eta_h^k = 0$. Moreover, given \eqref{eq::la_mi_ma} and the definitions of $\Gamma_1, \Gamma_2$, we can set 
	\begin{equation*}
		c_1 = \omega\min\limits_{h=1,\ldots,m}\lambda_\text{min}(\nabla^2 f_h(\bar{x})), \quad c_2 = \frac{1}{\omega}\max\limits_{h=1,\ldots,m}\lambda_\text{max}(\nabla^2 f_h(\bar{x}))
	\end{equation*}
	so that \cref{prop::vN_sdr} holds, i.e., $v_N(x_{j_k})$ satisfies the conditions at step \ref{step:sdr}. We then conclude that, at iteration $k$, the pure Newton's direction is employed.
	
	Now, from \cite[Corollary 5.2]{fliege2009newton} and \cite[Theorem 5.1]{fliege2009newton}, we obtain that
	\begin{equation*}
		F(x_{j_k} + v_N(x_{j_k})) \le F(x_{j_k}) + \boldsymbol{1}(2\gamma\theta_N(x_{j_k}))  \le F(x_{j_k}) + \boldsymbol{1}\gamma\mathcal{D}(x_{j_k}, v_N(x_{j_k}))
	\end{equation*}
	with $2\gamma < 1$ by assumption. Thus, $\alpha_0 = 1$ satisfies the sufficient decrease condition of the Armijo-type line search at step \ref{step:line_search_N}. Following a similar reasoning as in the proof of \cite[Theorem 5]{gonccalves2022globally}, we can use again \cite[Theorem 5.1]{fliege2009newton} to have $x_{j_{k + 1}} \in \mathcal{B}(x_{j_{\bar{k}}},r) \subsetneq \mathcal{N}(\bar{x})$. 
	
	Since $x_{j_{\bar{k}}} \in \mathcal{B}(x_{j_{\bar{k}}},r)$, we can proceed by inductive arguments to prove that items (i), (ii), (iii), (iv) hold for all $k \ge \bar{k}$. Moreover, this last result indicates that the sequence $\{x_{j_k}\}_{k \ge \bar{k}}$ coincides with the one generated by the Newton method \cite{fliege2009newton}. Therefore, all the convergence results directly follow from \cite[Theorem 5.1, Corollary 5.2, Theorem 6.1, Corollary 6.2]{fliege2009newton}.
\end{proof}

We conclude this section showing that the Barzilai-Borwein type direction from \cite{chen2023barzilai} makes the control at step \ref{step:sdr} irrelevant if $\Gamma_1$ and $\Gamma_2$ are suitably chosen.
\begin{proposition}
	\label{prop:BB}
	Assume $\Gamma_1, \Gamma_2$ are set in \cref{alg::F-Newton} such that $\Gamma_1\le \frac{a_{{\text{min}}}}{4a_{\text{max}}^2}$ and $\Gamma_2\ge \frac{1}{a_{{\text{min}}}}$, with $a_{\text{max}}\ge a_{\text{min}}>0$. Further assume that direction $v_D(x_c) = v_a(x_c)$ is computed according to \eqref{eq:bb_direction}, with $a_{\text{min}}\le a_j\le a_{\text{max}}$ for all $j=1,\ldots,m$. Then, $v_D(x_c)$ always passes the tests at step \ref{step:sdr}.
\end{proposition}
\begin{proof}
	By the definition of $v_a(x_c)$, we have
	\begin{align*}
		\max_{j=1,\ldots,m}\frac{\nabla f_j(x_c)^\top v_a(x_c)}{a_j}+\frac{1}{2}\|v_a(x_c)\|^2& = \min_{d\in\mathbb{R}^n} \max_{j=1,\ldots,m}\frac{\nabla f_j(x_c)^\top d}{a_j}+\frac{1}{2}\|d\|^2\\& \le \min_{d\in\mathbb{R}^n}	\max_{j=1,\ldots,m}\frac{\nabla f_j(x_c)^\top d}{a_{\text{max}}}+\frac{1}{2}\|d\|^2,
	\end{align*} 
	where the inequality follows from $a_j\le a_{\text{max}}$ for all $j$ and the fact that the two minima are both attained by directions such that $\nabla f_j(x_c)^\top d\le 0$ for all $j$. 
	Then, we can  write
\begin{gather*}
	\min_{d\in\mathbb{R}^n}	\max_{j=1,\ldots,m}\frac{\nabla f_j(x_c)^\top d}{a_{\text{max}}}+\frac{1}{2}\|d\|^2=\frac{1}{a_{\text{max}}}\min_{d\in\mathbb{R}^n}	\max_{j=1,\ldots,m}{\nabla f_j(x_c)^\top d}+\frac{a_{\text{max}}}{2}d^\top I d
	 = \frac{1}{a_{\text{max}}}\theta_{N}(x_c),
\end{gather*}
where $\theta_N(x_c)$ is obtained from \eqref{eq:newt-type-dirs} with $B_j = a_{\text{max}} I$ for all $j$. Recalling that $\theta_N(x)\le \mathcal{D}(x,v_N(x))/2$ and then \cite[Lemma 2.8]{lapucci2024convergence}, we get 

\begin{gather*}
	\frac{1}{a_{\text{max}}}\theta_{N}(x_c)\le \frac{1}{2a_{\text{max}}}\mathcal{D}(x_c,v_{N}(x_c))\le \frac{1}{2a_{\text{max}}}\left(\frac{-a_{\text{max}}}{2(a_{\text{max}})^2}\|v(x_c)\|^2\right)=-\frac{1}{4a_{\text{max}}^2}\|v(x_c)\|^2.
\end{gather*}
On the other hand:
	\begin{multline*}
		\max_{j=1,\ldots,m}\frac{\nabla f_j(x_c)^\top v_a(x_c)}{a_j}+\frac{1}{2}\|v_a(x_c)\|^2\ge
		\max_{j=1,\ldots,m} \frac{\nabla f_j(x_c)^\top v_a(x_c)}{a_j}\\\ge \max_{j=1,\ldots,m} \frac{\nabla f_j(x_c)^\top v_a(x_c)}{a_{{\text{min}}}}=\frac{1}{a_{\text{min}}}\mathcal{D}(x_c,v_a(x_c)),
	\end{multline*}
	where the second inequality comes from $v_a(x_c)$ being a descent direction and $a_j\ge a_{\text{min}}$ for all $j$.
	Putting everything together, we get
	\begin{equation}
		\label{eq:bb_proof_1}
		\mathcal{D}(x_c,v_a(x_c))\le- \frac{a_{{\text{min}}}}{4a_{\text{max}}^2}\|v(x_c)\|^2.
	\end{equation}

	Now, we note that $v_a(x_c)$ is the steepest descent direction for $(\frac{f_1(x)}{a_1},\ldots,\frac{f_m(x)}{a_m})^\top$ at $x_c$; thus, we have  
	$\theta_a(x_c) = -\frac{1}{2}\|v_a(x_c)\|^2$ (see, e.g., \cite[Eq.\ (16)]{gonccalves2022globally}).
	We also have
	\begin{align*}
		\theta_a(x_c)&=\max_{j=1,\ldots,m}\frac{\nabla f_j(x_c)^\top v_a(x_c)}{a_j}+\frac{1}{2}\|v_a(x_c)\|^2\\&\ge \frac{1}{a_{{\text{min}}}}\left(\max_{j=1,\ldots,m}\nabla f_j(x_c)^\top v_a(x_c) + \frac{a_{\text{min}}}{2}\|v_a(x_c)\|^2\right)\\&\ge \frac{1}{a_{{\text{min}}}}\min_{d\in\mathbb{R}^n}\max_{j=1,\ldots,m}\nabla f_j(x_c)^\top d + \frac{a_{\text{min}}}{2}d^\top Id = \frac{1}{a_{{\text{min}}}}\theta_{N}(x_c),
	\end{align*}
	where now $\theta_{N}(x_c)$ is obtained according to \eqref{eq:newt-type-dirs} with $B_j = a_{\text{min}}I$ for all $j$.
	From the solution of the dual problem, letting $\nu_1,\ldots, \nu_m$ the Lagrange multipliers s.t.\ $\nu_j \ge 0$ for all $j \in \{1,\ldots, m\}$ and $\mathbf{1}^\top\nu = 1$, we get (see, e.g., \cite[Eq.\ (10)]{gonccalves2022globally}) $$\theta_N(x_c) = -\frac{1}{2}v_N(x_c)^T\left[\sum_{j=1}^{m}\nu_j a_{\text{min}}I\right]v_N(x_c)=\frac{-a_{\text{min}}}{2}\|v_N(x_c)\|^2.$$
	We thus obtain
	\begin{align*}
		\theta_a(x_c)\ge \frac{1}{a_{{\text{min}}}}\theta_{N}(x_c) = \frac{1}{a_{{\text{min}}}}\left(-\frac{a_{\text{min}}}{2}\|v_N(x_c)\|^2\right) =  -\frac{1}{2}\|v_N(x_c)\|^2. 
	\end{align*}
	Therefore, we can finally write
	$\|v_a(x_c)\|^2 = -2\theta_a(x_c)\le \|v_N(x_c)\|^2,$
	i.e.,
	$$\|v_a(x_c)\|\le \|v_N(x_c)\|\le \frac{1}{a_{\text{min}}}\|v(x_c)\|,$$
	where the last inequality follows from \cite[Lemma 2.8]{lapucci2024convergence}. The above result, combined with \eqref{eq:bb_proof_1}, completes the proof.
\end{proof}

\subsection{Convergence Results for the Sequence of Sets of Points}
\label{sec:conv_front}
In this section, we are interested in analyzing the properties of the sequence $\{X^k\}$ considering the sets $X^k$ as a whole, rather than looking at their individual points. 

In order to do so, we need to give some definitions (for reference, see e.g.\ \cite{lacour2017box}).

\begin{definition}
	\label{def:HV}
	Let ${\zeta}\in\mathbb{R}^m$ be a reference point and let $Y\subseteq \mathbb{R}^m$ be a (possibly infinite) set of points. We define the \textit{dominated region} as 
	$$\Lambda(Y) = \left\{y\in\mathbb{R}^m\mid \exists\, \bar{y}\in Y:\; \bar{y}\le y\le {\zeta}\right\}.$$
	In addition, the \textit{hypervolume} associated to $Y$ is defined as the volume (or the Lebesgue measure) of the set $\Lambda(Y)$, i.e., $\mathcal{M}(\Lambda(Y))$, and is denoted by $V(Y)$.
\end{definition}
Note that if \rev{$X \subseteq \mathbb{R}^n$ is a set of points, $Y = F(X)$ is the image of $X$ through $F$ and $Y_\text{nd} = \{y \in F(X) \mid \nexists \bar{y} \in F(X) \text{ s.t.\ } \bar{y} \lneqq y\}$ is the image of the non-dominated points in $X$,} we have $\Lambda(Y_\text{nd}) = \Lambda(Y)$. Thus, the dominated region and the hypervolume of the \rev{image} set \rev{$Y$} only depend on the nondominated points in the set \rev{$X$}. We are thus particularly interested in looking at the hypervolume of \textit{stable sets}, i.e., \rev{image} sets of mutually nondominated points.

\rev{For the sake of notation simplicity, we will denote by $\Lambda_F(X)$ the dominated region $\Lambda(Y) = \Lambda(F(X))$. We will similarly denote the corresponding hypervolume $V(Y) = V(F(X))$ as $V_F(X)$.}

Hypervolume is often employed as a measure to compare Pareto front approximations in multi-objective optimization problems.
A graphical representation of dominated region and hypervolume for a bi-objective optimization problem is shown in \cref{fig::dr}. Note that, if the set $Y\subseteq\mathbb{R}^m$ is finite, the dominated region is obtained as the union of hyper-boxes. 

We will also exploit the following fundamental property of hypervolume, that generalizes \cite[Lemma 3.1]{custodio2021worst}. 
\begin{lemma}
	\label{lemma:HV}
	Let $\zeta\in\mathbb{R}^m$ be a reference point and let \rev{$X$ be a set of points such that $Y = F(X)$ is a stable set and, for all $y\in Y$, $y\le \zeta$}. Let \rev{$\bar{x}\in X$} and $\mu\in\mathbb{R}^\rev{n}$ s.t.\ \rev{$F(\mu)<F(\bar{x})$}. Then, the set $Z = \rev{F(X\cup\{\mu\}\setminus\{\bar{x}\})}$ is such that $$V(Z)-V(Y)\ge \prod_{j=1}^{m}(\rev{f_j(\bar{x})-f_j(\mu)})>0.$$
\end{lemma}
\begin{proof}
	We start by noting that, being \rev{$F(\mu)<F(\bar{x})$} and since the dominated region and the hypervolume of a\rev{n image} set only depend on \rev{the} nondominated solutions, we have $\Lambda(Z)=\Lambda(\hat{Z})$ and $V(Z) = V(\hat{Z})$, with $\hat{Z} = \rev{F(X\cup\{\mu\})}$.
	Now, we prove two properties.
	\begin{enumerate}[(a)]
		\item If $w\in\Lambda(Y)$, then $w\in\Lambda(\hat{Z})$ - i.e., $\Lambda(\hat{Z})\supseteq\Lambda(Y)$ and thus measure of the dominated region cannot decrease if a new nondominated point is added. 
		\item Let \rev{$[F(\mu),F(\bar{x})) = [f_1(\mu),f_1(\bar{x}))\times\ldots\times[f_m(\mu),f_m(\bar{x}))$}. \rev{We have, f}or all $w\in[\rev{F(\mu),F(\bar{x})})$, $w\notin \Lambda(Y)$ and $w\in\Lambda(\hat{Z})$ - i.e., the hyperrectangle $[\rev{F(\mu),F(\bar{x})})$ is entirely included in the new dominated region $\Lambda(\hat{Z})$, whereas it was entirely excluded by the old dominated region $\Lambda(Y)$.
	\end{enumerate}
	We prove the two properties one at a time.
	\begin{enumerate}[(a)]
		\item Let $w\in \Lambda(Y)$; then, by the definition of $\Lambda$, $w\le \zeta$ and there exists $\rev{x'\in X}$ such that $\rev{F(x') = y'}\le w$. Since $y'\in Y\subset \hat{Z}$, we have $y'\le w\le \zeta$ with $y'\in \hat{Z}$, and the property is thus proved.   
		\item By definition, for any $w\in[\rev{F(\mu),F(\bar{x})})$ we have $\rev{F(\mu)}\le w$ and $w<\rev{F(\bar{x})}$; moreover $\rev{F(\bar{x})}\le \zeta$; thus, $\rev{F(\mu)}\le w\le \zeta$, i.e., $w\in \Lambda(\hat{Z})$.
		
		Now, $Y$ is stable, so for any $\rev{x'\in X}$ there must exist $j\in\{1,\ldots,m\}$ s.t.\ $ \rev{f_j(x')\ge f_j(\bar{x})}$ and then $w_j<\rev{f_j(\bar{x})}\le \rev{f_j(x')}$. Thus, there is no $\rev{x' \in X}$ such that $\rev{F(x') =}\ y' \le w$ and then $w\notin \Lambda(Y)$.   
	\end{enumerate}
	Now, by property (a) we get that $$V(Z)-V(Y) = V(\hat{Z})-V(Y) = \mathcal{M}(\Lambda(\hat{Z})\setminus \Lambda(Y)).$$
	From property (b) we know that  $[\rev{F(\mu),F(\bar{x})})\subseteq \Lambda(\hat{Z})\setminus \Lambda(Y)$. Thus, we can write
	$$\mathcal{M}(\Lambda(\hat{Z})\setminus \Lambda(Y))\ge \mathcal{M}([\rev{F(\mu),F(\bar{x})})) = \prod_{j=1}^{m}(\rev{f_j(\bar{x})-f_j(\mu)}).$$
	Combining the last two equations we get the thesis.
\end{proof}
The result from above Lemma is graphically explained in \cref{fig::dr_lemma}.

\begin{figure}
	\subfloat[The dominated region $\Lambda(Y)$ for the set $Y$, given the reference point $\zeta$. The hypervolume of $Y$ is the area of $\Lambda(Y)$.\label{fig::dr}]{\includegraphics[width=0.49\textwidth]{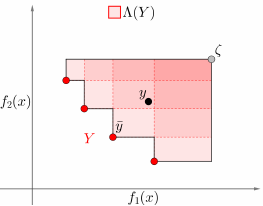}}
	\hfil
	\subfloat[\rev{Given a stable set $Y = F(X)$, when a point $\mu$ is added that dominates $\bar{x}\in X$, the dominated region enlarges to at least contain the full hyperbox $[F(\mu),F(\bar{x}))$.\label{fig::dr_lemma}}]{\includegraphics[width=0.49\textwidth]{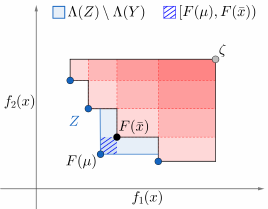}}
	\caption{Graphical representation of \cref{def:HV} and \cref{lemma:HV}.}
	\label{fig::dr_gen}
\end{figure}

Before turning to the convergence analysis, we give a last definition that allows to introduce a novel Pareto-stationarity measure for sets of mutually nondominated points.
\begin{definition}
	Let $\mathcal{X}$ be the set of all sets  $X\subseteq\mathbb{R}^n$ of mutually nondominated points w.r.t.\ $F$, i.e., $X\in \mathcal{X}$ if $F(X)$ is a stable set. We define the map $\Theta:\mathcal{X}\to\mathbb{R}$ as
	$$\Theta(X) = \inf_{x\in X}\theta(x).$$ 
\end{definition}
Note that, for a set $X$ with finite cardinality, the infimum is in fact a minimum, i.e., there exists $x\in X$ such that $\Theta(X)=\theta(x)$. In this latter case, function $\Theta(\cdot)$ thus returns the common steepest descent value $\theta(\cdot)$ of the ``least'' Pareto-stationary point\rev{s} in the set $X$. $\Theta(X)$ of course always takes nonpositive values and is zero if and only if all the points in $X$ are Pareto-stationary.
At this point, we need to state an assumption on the implementation of the algorithm.
\begin{assumption}
	\label{ass:order}
	 At each iteration $k$ of \cref{alg::F-Newton},  the first point to be processed in the for loop of steps \ref{step:main_for}-\ref{step:end_loop} belongs to the set
	$\argmin_{x\in X^k}\theta(x)$. 
\end{assumption}

We are finally able to provide a convergence result concerning the asymptotic “Pareto-stationarity”
of the sets $\{X^k\}$.
\begin{theorem}
	\label{prop:big_theta}
	Let $X^0$ be a set of mutually nondominated points and $x_0\in X^0$ be a point such that the set $\mathcal{L}(x_0) = \bigcup_{j=1}^{m}\left\{x\in\mathbb{R}^n\mid f_j(x)\le f_j(x_0)\right\}$ is compact.
	Let $\left\{X^k\right\}$ be the sequence of sets of nondominated points produced by \cref{alg::F-Newton} under \cref{ass:order}. 
	Then, 
	\begin{enumerate}[(i)]
		\item if $\sigma_k=\sigma>0$ for all $k$, there exists $\bar{k}$ such that $\Theta(X^{{k}})\ge -\sigma$ for all $k\ge \bar{k}$;
		\item  if $\sigma_k\to 0$, $\lim_{k\to \infty}\Theta(X^k) = 0$.
	\end{enumerate}
\end{theorem}

\begin{proof}
	We begin the proof showing that a reference point $\zeta\in\mathbb{R}^m$ exists such that $F(\xi)\le \zeta$ for all $\xi\in X^k$, for all $k\in\{0,1,\ldots\}$.
	To this aim, let 
	$\bar{\zeta}\in\mathbb{R}^m$ such that $\bar{\zeta}_j = \max_{x\in\mathcal{L}(x_0)}f_j(x),$
	which is well-defined being $F$ continuous and $\mathcal{L}(x_0)$ compact. 
	Let $k$ be any iteration and $\xi\in X^k$. By \cref{lemma:compact_set} we have $\xi\in\mathcal{L}(x_0)$. Therefore, for all $h\in\{1,\ldots,m\}$, we have 
	$f_h(\xi)\le \max_{x\in\mathcal{L}(x_0)}f_h(x) = \bar{\zeta}_h,$
	and thus 
	\begin{equation}
		\label{eq:reference}
		F(\xi)\le \bar{\zeta}\quad\text{ for all } \xi\in X^k,\;\text{for all } k.
	\end{equation}
	
	Now let us assume by contradiction that the thesis is false, in the two cases $\sigma_k=\sigma>0$ for all $k$ and $\sigma_k\to0$, respectively:
	\begin{enumerate}[(i)]
		\item there exists an infinite subsequence $K\subseteq\{0,1,\ldots\}$ such that, \rev{for all $k\in K$ sufficiently large}, we have $\Theta(X^k)<-\sigma$;
		\item there exists an infinite subsequence $K\subseteq\{0,1,\ldots\}$ and a value $\epsilon>0$ such that, \rev{for all $k\in K$  sufficiently large}, we have $\Theta(X^k)<-\epsilon$.
	\end{enumerate}
	We can thus analyze the two cases at once, assuming that there exists an infinite subsequence $K\subseteq\{0,1,\ldots\}$ and a value $\bar{\sigma}>0$ such that, \rev{for all $k\in K$ and sufficiently large}, we have $\Theta(X^k)<-\bar{\sigma}$. {Since either $\bar{\sigma}=\sigma_k$ from case (i) or $\sigma_k\to0$ from case (ii), we also have that $\Theta(X^k)<-\sigma_k$ for all $k\in K$ sufficiently large.}
	
	\rev{Now, let $\{x^k\}$ be the sequence of points such that, for all $k$, $x^k$ is the first point to be processed in the for loop of steps \ref{step:main_for}-\ref{step:end_loop}. By  \cref{ass:order},
	$x^k\in\argmin_{x\in X^k}\theta(x),$
	i.e., $\{x^k\}$ is a sequence of ``least Pareto-stationary points'' in $\{X^k\}$, with $\theta(x^k) = \Theta(X^k)$. 
    } By \cref{lemma:compact_set}, we know that $\{x^k\}\subseteq\mathcal{L}(x^0)$, and so does $\{x^k\}_K$ which thus admits accumulation points.

	Let $\bar{x}$ be such an accumulation point, i.e., there exists $K_1\subseteq K$ such that $x^k\to \bar{x}$ for $k\in K_1$, $k\to \infty$. 
	Let us define $z^k = x^k + \alpha_kd_k$ as the point obtained at step \ref{step:z_N} of the algorithm while processing point $x^k$; by the continuity of $v(\cdot)$, we get $v(x^k)\to v(\bar{x})$ for $k\in K_1$, $k\to\infty$.
	By the second condition in the control at step \ref{step:sdr} of the algorithm, either $d_k = v(x^k)$ or $\|d_k\|\le \Gamma_2 \|v(x^k)\|$; therefore, $\|d_k\|\le \max\{1,\Gamma_2\}\|v(x^k)\|$. Since $v(x^k)\to v(\bar{x})$, the sequence $d_k$ is bounded for $k\in K_1$. Moreover, $\alpha_k\in[0,\alpha_0]$, which is a compact set. Therefore, there exists a further subsequence $K_2\subseteq K_1$ such that, for $k\in K_2$, $k\to\infty$,
	$$\lim_{\substack{k\to\infty\\k\in K_2}}\alpha_k= \bar{\alpha}\in[0,\alpha_0],\qquad\qquad \lim_{\substack{k\to\infty\\k\in K_2}}d_k= \bar{d},\qquad\qquad \lim_{\substack{k\to\infty\\k\in K_2}}z^k= \bar{x} + \bar{\alpha}\bar{d} = \bar{z}.$$
	
	\noindent By the definition of $\{x^k\}$, $K$ and $\{\sigma_k\}$, we have for all $k\in K$ sufficiently large that $\theta(x^k) = \Theta(X^k)<-\bar{\sigma}\le -\sigma_k\le0$. Then, $\alpha_k$ has certainly be obtained, at step \ref{step:line_search_N}, by the line search along the direction $d_k$, and we have $$F(z^k)\le F(x^k) +\boldsymbol{1} \gamma \alpha_k \mathcal{D}(x^k,d_k).$$
	
	By the instructions of the algorithm (in particular from the first condition at step \ref{step:sdr}) we have $\mathcal{D}(x^k,d_k)\le -\Gamma_1\|v(x^k)\|^2$ or $\mathcal{D}(x^k,d_k) = \mathcal{D}(x^k,v(x^k))= -\|v(x^k)\|^2$, therefore
	\begin{equation}
		\label{eq:front-suf-dec}
		F(z^k)\le F(x^k) - \boldsymbol{1}\min\left\{1,\Gamma_1\right\}\gamma \alpha_{k} \|v(x^k)\|^2.
	\end{equation}

	By the continuity of $\theta$, we also have, taking the limits for $k\in K_2$, $k\to \infty$, $$\theta(\bar{x}) = \lim_{\substack{k\to\infty\\k\in K_2}}\theta(x^k)\le -\bar{\sigma}<0.$$
	Then, $\bar{x}$ is not Pareto-stationary and $\|v(\bar{x})\|=\eta_1>0$. Then, by the continuity of $v$ and of the norm function, there exists a value $\eta_2$ such that $0 < \eta_2\le \eta_1$ and $\|v(x^k)\|\ge \eta_2$ for all $k\in K_2$ sufficiently large. 
	Plugging this inequality into \cref{eq:front-suf-dec}, we get for $k\in K_2$ sufficiently large
	\begin{equation}
		\label{eq:suff_dec_front_eps}
		F(z^k)\le F(x^k)-\boldsymbol{1}\min\left\{1,\Gamma_1\right\}\gamma \alpha_k\eta_2^2.
	\end{equation}
	
	We now consider the behavior of the hypervolume metric along the sequence of sets $\{X^k\}$. 
	By the instructions of the algorithm, a point is removed from the current set of solutions only when a new point is added that dominates it. In other words, if $\xi\in X^k$ and $\xi\notin X^{k+1}$, there exists $\xi_1\in X^{k+1}$ such that $F(\xi_1)\lneqq F(\xi)$. Therefore, $(F(X^{k+1}\cup X^k))_\text{nd} = (F(X^{k+1}))_\text{nd} = F(X^{k+1})$, where the last equality follows from the nondominance property of $\{X^k\}$ (\cref{lemma:prelims_N}); moreover, it is straightforward to observe that, since $X^k\subseteq(X^{k+1}\cup X^k)$, we have $\Lambda_F(X^k)\subseteq\Lambda_F(X^{k+1}\cup X^k)$.
	Thus, by the properties of the dominated region, we have:
	$$\Lambda_F(X^{k+1}) = \Lambda((F(X^{k+1}\cup X^k))_\text{nd}) = \Lambda_F(X^{k+1}\cup X^k) \supseteq \Lambda_F(X^k).$$
	Then, we can also write
	$V_F(X^{k+1})\ge V_F(X^k),$
	i.e., the sequence $\{V_F(X^k)\}$ is monotone nondecreasing and thus \rev{it either diverges to $+\infty$ or} admits limit $\bar{V}$. We can note that \rev{this is, in fact, the second case, with $\bar{V}$ being finite. I}ndeed, similarly as $\bar{\zeta}$, let $\pi\in\mathbb{R}^m$ be such that $\pi_j = \min_{x\in\mathcal{L}(x_0)} f_j(x)$. We then observe that $\Lambda_F(X^k)\subseteq\{y\in\mathbb{R}^m\mid \pi\le y\le \bar{\zeta}\}$, the latter set being a compact set and thus having a finite measure $M$. Therefore, $V_F(X^k)\le M<\infty$ \rev{and the sequence $\{V_F(X^k)\}$ must be then convergent to a finite limit $\bar{V}$.}

	Now, by \cref{lemma:lemma_N}, a point $y^k$ exists, for all $k$, such that $y^k\in X^{k+1}$ with $F(y^k)\le F(z^k)$.
	We have $(X^{k+1}\cup X^k)\supseteq X^k\cup \{y^k\}\supseteq X^k$, so that,
	by the properties of the dominated region, we can also write:
	$$\Lambda_F(X^{k+1})=\Lambda_F(X^{k+1}\cup X^k)\supseteq \Lambda_F(\{y^{k}\}\cup X^k)\supseteq \Lambda_F(X^k).$$
	and thus
	\begin{equation}
		\label{eq:hv_inc}
		V_F(X^{k+1}) \ge V_F(\{y^{k}\}\cup X^k) \ge V_F(X^k).
	\end{equation}
	We shall now note that $F(y^k)\le F(z^k)<F(x^k)$; thus, $x^k$ is dominated by $y^k$ and
	\begin{equation}
		\label{eq:hv_same}
		V_F(\{y^{k}\}\cup X^k) = V_F(\{y^{k}\}\cup X^k\setminus\{x^k\}).
	\end{equation} 
	We shall also note that, by \cref{lemma:prelims_N}, for all $k$ the set $X^k$ contains mutually nondominated points, i.e., $F(X^k)$ is a stable set. Recalling \cref{eq:reference}, we have that all the assumptions of \cref{lemma:HV} are satisfied when we evaluate $V_F(\{y^{k}\}\cup X^k\setminus\{x^k\})$. In particular, we get
	\begin{equation}
		\label{eq:hv_chain}
		V_F(\{y^{k}\}\cup X^k\setminus\{x^k\})-V_F(X^k)\ge \prod_{j=1}^m(f_j(x^k)-f_j(y^k))\ge \prod_{j=1}^m(f_j(x^k)-f_j(z^k)).
	\end{equation}
	Putting together \cref{eq:suff_dec_front_eps}, \cref{eq:hv_inc}, \cref{eq:hv_same} and \cref{eq:hv_chain}, we finally obtain that 
	$$V_F(X^{k+1})-V_F(X^k)\ge V_F(\{y^{k}\}\cup X^k\setminus\{x^k\}) - V_F(X^k) \ge \left(\min\left\{1,\Gamma_1\right\}\gamma \alpha_k\eta_2^2\right)^m.$$ 
	Recalling that $V_F(X^k)\to \bar{V}$, we can take the limits for $k\in K_2$, $k\to\infty$ to obtain
	$$\lim_{\substack{k\to\infty\\k\in K_2}}\left(\min\left\{1,\Gamma_1\right\}\gamma \alpha_k\eta_2^2\right)^m \le 0.$$
	Since $\min\left\{1,\Gamma_1\right\}>0$, $\eta_2>0$ and $\gamma>0$, we necessarily have that $\alpha_k\to 0$ for $k\in K_2$, $k\to \infty$.

	Since $\alpha_k$ is defined at step \ref{step:line_search_N} and, for $k\in K_2$, $\alpha_{k}\to 0$, given any $q\in\mathbb{N}$, for all $k\in K_2$ large enough we necessarily have $\alpha_{k}< \alpha_0\delta^q$. The step $\alpha=\alpha_0\delta^q$ hence does not satisfy the Armijo condition $F\left(x^k+\alpha d_k\right)\le F(x^k)+\boldsymbol{1}\gamma \alpha \mathcal{D}(x^k, d_k)$, i.e., for some $\tilde{h}(k)$ we have $$f_{\tilde{h}(k)}\left(x^k+\alpha_0\delta^q d_k\right) > f_{\tilde{h}(k)}(x^k)+\gamma \alpha_0\delta^q \mathcal{D}(x^k, d_k).$$
	Taking the limits along a suitable subsequence such that $\tilde{h}(k)= \tilde{h}$, we get $$f_{\tilde{h}}\left(\bar{x}+\alpha_0\delta^q\bar{d}\right)\ge f_{\tilde{h}}(\bar{x})+\gamma \alpha_0\delta^q\mathcal{D}(\bar{x}, \bar{d}).$$
	Being $q$ arbitrary and recalling \cite[Lemma 4]{fliege2000steepest}, it must be $\mathcal{D}(\bar{x},\bar{d})\ge 0$. 
	Yet, we know that $\mathcal{D}(x^k,d_k)\le -\min\left\{1,\Gamma_1\right\}\|v(x^k)\|^2$. In the limit we obtain
	$\mathcal{D}(\bar{x},\bar{d})\le -\min\left\{1,\Gamma_1\right\}\|v(\bar{x})\|^2<0$. 
	We finally get a contradiction, which completes the proof.	
\end{proof}

\begin{remark}
	\rev{Although  \cref{ass:order} may initially appear strong for establishing results over the entire sequence of sets, it is necessary to prevent a specific (and practically unlikely) degenerate behavior of the algorithm: a point in the population never undergoing the refinement step because it is consistently dominated, yet without exhibiting sufficient decrease, before being considered for optimization. Specifically, let $x^k\in X^k$ be one of the ``least Pareto-stationary'' points in the current set, i.e., $\theta(x^k) = \Theta(X^k)$. Without making further assumptions, during iteration $k$ a point $y'$ dominating $x^k$ might be found while processing some other point $x'\in X^k$ such that $x' \not\in \argmin_{x \in X^k}\theta(x)$. The point $x^k$ would be filtered out of $\hat{X}^k$ before getting processed, i.e., before doing the corresponding refinement step. The new point $y'$, however, might not provide a sufficient decrease of $F$ w.r.t.\ $F(x^k)$. Iterating this situation, we would end up with a sequence of points that is steadily improved, but at a too slow pace to reach stationarity.
		
	Furthermore,  \cref{ass:order} can arguably be considered not restrictive for \cref{alg::F-Newton}. Indeed, once the gradients and values of $\theta$ are computed at the beginning of iteration $k$ to search for the $\argmin_{x\in X^k}\theta(x)$, this information can be stored and reused later in the iteration when processing the points in $X^k$. In the end, the increase in the computational cost is only represented by the computation of the gradients of points that get filtered out of the list before getting processed. In later iterations especially, the latter is a rather marginal circumstance.
	
	We also note that our above claim is supported by the robust numerical evidence presented later in \cref{sec::computational_experiments}, where the experiments were conducted with the assumption embedded in the algorithm implementation.}
\end{remark}

\begin{remark}
	\label{rem:hypervol}
	Point (ii) of \cref{prop:big_theta} \rev{obviously} implies that, when $\sigma_k\to 0$, for any tolerance $\epsilon>0$ we reach $\Theta(X^k)>-\epsilon$ in a finite number of iterations. By point (i) of \cref{prop:big_theta}, this also holds true when $\sigma_k$ is a positive constant $\sigma$, as long as $\epsilon\ge\sigma$. We could thus set a threshold $\epsilon$ on the value of $\Theta(X^k)$ to define a stopping condition for the algorithm. However, such a stopping condition might be activated when there is still room to improve the spread of the solution. 
	We shall thus observe, by similar reasonings as those in the proof of \cref{prop:big_theta}, that the improvement in the hypervolume $V_F(X^k)$ also goes to zero. The stopping condition 
	\begin{equation}
		\label{eq::hypervolume}
		\frac{V_F(X^{k+1})-V_F(X^k)}{V_F(X^k)}<\varepsilon_{hv}
	\end{equation}
	is therefore enough to guarantee finite termination of the algorithm and it might be more suitable in practice. 
\end{remark}

By \cref{prop:big_theta} we immediately deduce that, under the additional requirement of \cref{ass:order}, we can obtain a stronger result concerning the convergence of any sequence of points $\{x^k\}$ such that $x^k\in X^k$, and not just linked sequences.
\begin{corollary}
	Let $\left\{X^k\right\}$ be the sequence of sets of nondominated points produced by \cref{alg::F-Newton} under the assumptions of \cref{prop:big_theta} and with $\sigma_k\to0$. Let $\left\{x^k\right\}$ be any sequence such that $x^k\in X^k$ for all $k$, then it admits accumulation points and every accumulation point is Pareto-stationary for problem \cref{eq:mo_prob}.
\end{corollary}
\begin{proof}
	From \cref{lemma:compact_set}, we have that $\{x^k\}\subseteq \mathcal{L}(x_0)$ and thus admits cluster points.
	By the definition of $\Theta(\cdot)$, we know that
	$0\ge \theta(x^k)\ge \Theta(X^k)$,
	for all $k$. Taking the limit along any subsequence $K$ such that $x^k\to \bar{x}$ for $k\in K$, $k\to\infty$, recalling \cref{prop:big_theta}, we have $\theta(\bar{x}) =0$. This completes the proof.
\end{proof}

Similarly as for sequences of points, we can accompany the asymptotic convergence result with an interesting worst case complexity bound for the number of iterations required to drive $\Theta(X^k)$ above a given threshold. 
\begin{theorem}
	\label{prop:complexity_set}
	Assume the gradients $\nabla f_1(x),\ldots,\nabla f_m(x)$ are Lipschitz continuous with constants $L_1,\ldots,L_m$ ($L_\text{max} = \max_{j=1,\ldots,m}L_j$). Let $\{X^k\}$ be the sequence generated by \cref{alg::F-Newton} under the assumptions of \cref{prop:big_theta} and let $V^*=V_F(\mathcal{L}(x_0))$ (which is a finite value by the compactness of $\mathcal{L}(x_0)$). 
	Then, $\{X^k\}$ is such that, for any $\epsilon$ such that	
	\begin{enumerate}[(i)]
		\item $\epsilon\ge \sigma$ in the case $\sigma_k=\sigma>0$ for all $k$, or
		\item $\epsilon>0$ in the case $\sigma_k\to 0$,
	\end{enumerate} 
  	at most $k_\text{max}$ iterations are needed to produce an iterate set $X^k$ such that $\Theta(X^k)>-\epsilon$, where 
	$$k_\text{max}\le \frac{V^*-V_F(X^0)}{(2\gamma\min\{1,\Gamma_1\}\min\{\alpha_0,{\Delta_\text{low}}\}\epsilon)^m}= \mathcal{O}\left(\frac{1}{\epsilon^{m}}\right),$$
	where $\Delta_\text{low} = \frac{\Gamma_1(1-\gamma)}{\Gamma_2^2L_\text{max}}$. 
	Moreover, the total number of iterations $N_{\text{it}_\epsilon}$ such that $\Theta(X^k)\le-\epsilon$ is also bounded by
	$$N_{\text{it}_\epsilon}\le \frac{V^*-V_F(X^0)}{(2\gamma\min\{1,\Gamma_1\}\min\{\alpha_0,{\Delta_\text{low}}\}\epsilon)^m}= \mathcal{O}\left(\frac{1}{\epsilon^{m}}\right).$$
\end{theorem}
\begin{proof}
	Similarly as in the proof of \cref{prop:big_theta}, \rev{let $\{x^k\}$ be the sequence such that, for all $k$, $x^k$ is the first point to be processed in the for loop of steps \ref{step:main_for}-\ref{step:end_loop}. By  \cref{ass:order},
	$x^k\in\argmin_{x\in X^k}\theta(x)$.} Let $z^k=x^k+\alpha_kd_k$ be the corresponding point obtained at step \ref{step:z_N}, which is well defined thanks to \cref{ass:order}. We then have
	\begin{equation*}
		F(z^k)\le F(x^k)-\boldsymbol{1}\min\{1,\Gamma_1\}\gamma \alpha_k\|v(x^k)\|^2.
	\end{equation*}
	Repeating the reasoning of the proof of \cref{prop:big_theta},
	we immediately get
	\begin{equation*}
		V_F(X^{k+1})-V_F(X^k)\ge \left(\min\{1,\Gamma_1\}\gamma\alpha_k \|v(x^k)\|^2\right)^m.
	\end{equation*}
	
	By \cite[Corollary 5.2]{lapucci2024convergence}, we are guaranteed that $\alpha_k\ge \min\{\alpha_0,{\Delta_\text{low}}\}$. Moreover, let us assume that for the first $\rev{k_\text{max}}$ iterations we have $\Theta(x^k)\le -\epsilon$ and thus 
	\begin{equation*}
		\|v(x^k)\|^2 = -2\theta(x^k) = -2\Theta(X^k)\ge2\epsilon.
	\end{equation*}
	Then, we obtain
	\begin{equation*}
		V_F(X^{k+1})-V_F(X^k)\ge \left(2\gamma\min\{1,\Gamma_1\}\min\{\alpha_0,{\Delta_\text{low}}\}\epsilon\right)^m.
	\end{equation*}
	
	Since, by \cref{lemma:compact_set}, $X^{k}\subseteq \mathcal{L}(x^0)$ for all $k$, we know that $V_F(X^k)$ is always bounded above by the value $V^*$. We can therefore write
	\begin{align*}
		V^*-V_F(X^0)&\ge V_F(X^{\rev{k_\text{max}}})-V_F(X^0)= \sum_{k=0}^{\rev{k_\text{max}}-1}V_F(X^{k+1})-V_F(X^k)\\&\ge\sum_{k=0}^{\rev{k_\text{max}}-1} \left(2\gamma\min\{1,\Gamma_1\}\min\{\alpha_0,{\Delta_\text{low}}\}\epsilon\right)^m\\& = \rev{k_\text{max}}\left(2\gamma\min\{1,\Gamma_1\}\min\{\alpha_0,{\Delta_\text{low}}\}\epsilon\right)^m.
	\end{align*}
	We finally get that 
	\begin{equation*}
		\rev{k_\text{max}}\le \frac{V^*-V_F(X^0)}{(2\gamma\min\{1,\Gamma_1\}\min\{\alpha_0,{\Delta_\text{low}}\}\epsilon)^m} = \mathcal{O}\left(\frac{1}{\epsilon^m}\right).
	\end{equation*}	
	
	Now, by \cref{prop:big_theta}, we know that both in cases (i) and (ii)  there exists $\bar{k}$ such that $\Theta(X^k)>-\epsilon$ for all $k\ge \bar{k}$. Recalling that $\{V_F(X^k)\}$ is \rev{a nondecreasing} sequence, this time we can write 
	\begin{align*}
		V^*-V_F(X^0)&\ge V_F(X^{\bar{k}})-V_F(X^0)= \sum_{k=0}^{\bar{k}-1}V_F(X^{k+1})-V_F(X^k)
		\\&\ge \sum_{k:\Theta(X^k)\le-\epsilon}^{}V_F(X^{k+1})-V_F(X^k)
		\\&\ge\sum_{k:\Theta(X^k)\le-\epsilon}^{} \left(2\gamma\min\{1,\Gamma_1\}\min\{\alpha_0,{\Delta_\text{low}}\}\epsilon\right)^m
		\\& = N_{\text{it}_\epsilon}\left(2\gamma\min\{1,\Gamma_1\}\min\{\alpha_0,{\Delta_\text{low}}\}\epsilon\right)^m.
	\end{align*}
	We can thus conclude that
	\begin{equation*}
		N_{\text{it}_\epsilon}\le \frac{V^*-V_F(X^0)}{\left(2\gamma\min\{1,\Gamma_1\}\min\{\alpha_0,{\Delta_\text{low}}\}\epsilon\right)^m},
	\end{equation*}
	\rev{which completes the proof.}
\end{proof}
{The interesting aspect with the above result is that complexity increases with the number of objectives of the problem. This appears reasonable, as we are not considering sequences of points, but set of points related to the Pareto front, i.e., a typically $(m-1)$-dimensional object. Also, notice that an $\mathcal{O}(\frac{1}{\epsilon^m})$ complexity result for $\Theta$ corresponds to an $\mathcal{O}(\frac{1}{\epsilon^{2m}})$ bound for the values of $\|v(x)\|$ of points in $X^k$. This is in line with $\mathcal{O}(\frac{1}{\epsilon^2})$ bounds on $\|\nabla f(x^k)\|$ typical of single objective ($m=1$) optimization.}

\begin{remark}
	\label{rem:finite_ref}
	The above theoretical results lead to some very insightful observations concerning the behavior of the algorithm in the computationally reasonable case where $\sigma_k$ is set to a nonzero constant $\sigma>0$. Indeed, point (i) of \cref{prop:big_theta} guarantees us that, if we set a tolerance on Pareto-stationarity of the points in the list of solutions, not only the algorithm at some point will stop carrying out refinement steps, but also that exploration steps will always provide points that are, straight away, sufficiently stationary according to the predefined threshold. \cref{prop:complexity_set} then tells us that the number of iterations where refinement steps are carried out is bounded by an $\mathcal{O}(\frac{1}{\sigma^m})$ quantity.
\end{remark}

\section{Computational Experiments}
\label{sec::computational_experiments}
In this section, we report the results of thorough computational experiments where we tested possible exemplars of the \textit{Front Descent} (FD) method, as well as some state-of-the-art approaches. All the code for the experiments was developed in \texttt{Python3}\footnote{The code of the proposed approaches is available at \href{https://github.com/pierlumanzu/fd_framework}{github.com/pierlumanzu/fd\_framework}.}; the experiments were run on a computer with the following characteristics: Ubuntu 22.04, Intel Xeon Processor E5-2430 v2 6 cores 2.50 GHz, 32 GB RAM. When a solver is needed to handle subproblems of the form \eqref{eq:newt-type-dirs}-\eqref{eq:bb_direction}, we employ \texttt{Gurobi} optimizer (version 10).

We considered the following types of direction in the refinement step of \cref{alg::F-Newton}: steepest common descent direction (FD-SD) \cite{lapucci2023improved}, Newton direction (FD-N) \cite{gonccalves2022globally}, limited memory Quasi-Newton direction (FD-LMQN) \cite{lapucci2023limited} and Barzilai-Borwein direction (FD-BB) \cite{chen2023barzilai}. For all variants of FD, we employed the following experimental setting: $\alpha_0=1$, $\delta=0.5$, $\gamma=10^{-4}$, $\Gamma_1=10^{-2}$, $\Gamma_2=10^2$, $\sigma_k=10^{-7}$ for all $k$. We also employed a heuristic based on crowding distance \cite{deb2002fast} to avoid the generation of too many close points. Using the same mechanism as in \cite{gonccalves2022globally}, at each iteration of FD-N we forced the eigenvalues of the Hessian matrices to be greater than $\rho = 10^{-2}$ so as to have the matrices $B_j$ definite positive for all $j$. In FD-LMQN we considered memory size equal to $5$, while in FD-BB the scalars $a_j$ used to rescale the gradients were bounded in $[10^{-3}, 10^{3}]$. 

As for the state-of-the-art approaches for Pareto front reconstruction in continuous unconstrained MOO, we considered the \rev{very popular evolutionary method NSGA-II \cite{deb2002fast}, the convergent derivative-free DMulti-MADS method \cite{Bigeon2021} (abbreviated here as DM-MADS) and the derivative-based MO Trust-Region approach (MOTR) \cite{mohammadi2024trust}}. For \rev{the three} algorithms, we set the parameters values as indicated in their respective papers. Since NSGA-II is non-deterministic, it was run 5 times with 5 different seeds for the pseudo-random number generator: only the best execution (chosen based on the \textit{Purity} metric \cite{custodio2011direct}) was considered for comparisons.

All algorithms were evaluated on a collection of unconstrained problems, both convex (JOS\_1 \cite{jin2001dynamic,lapucci2023limited}, the rescaled version of MAN\_1 \cite{lapucci2023memetic} proposed in \cite{lapucci2023limited}, SLC\_2 \cite{schutze2011directed}, MOP\_7 \cite{huband06}) and non-convex (MMR\_5 \cite{MIGLIERINA2008662}, the rescaled version of MOP\_2 \cite{huband06} reported in \cite{lapucci2023limited}, MOP\_3 \cite{huband06}, CEC09\_1, CEC09\_2, CEC09\_3, CEC09\_7, CEC09\_8, CEC09\_10 \cite{zhang2008multiobjective}). Most of the problems are bi-objective, while a small portion (MOP\_7, CEC09\_8, CEC09\_10) presents three objective functions; we refer the reader to the cited papers for more details on the formulations. \rev{We} varied the number of variables \rev{in the following way:}
\rev{\begin{itemize}
		\item JOS\_1, MAN\_1, SLC\_2, MMR\_5, MOP\_2: $n \in  \{2, 3, 4, 5, 6, 8, 10, 12, 15, 17,\allowbreak 20, 25, 30, 35, 40, 45, 50, 100, 200\}$;
		\item CEC09\_1, CEC09\_2, CEC09\_3, CEC09\_7: $n \in  \{4, 5, 6, 8, 10, 12, 15, 17, 20,\allowbreak 25, 30, 35, 40, 45, 50, 100, 200\}$;
		\item CEC09\_8, CEC09\_10: $n \in  \{5, 6, 8, 10, 12, 15, 17, 20, 25, 30, 35, 40, 45, 50, 100,\allowbreak 200\}$;
		\item MOP\_3, MOP\_7: $n = 2$.
\end{itemize}}
For each problem, \rev{the selection of starting solutions was deterministic:} $n$ points were uniformly sampled as starting solutions from the hyper-diagonal of a hyperbox, defined by lower and upper bounds suggested in the reference paper of each problem.

Finally, in order to summarize the results of the experiments, we made use of \textit{performance profiles} \cite{dolan2002benchmarking}. 
We mainly considered standard metrics from MOO literature: \textit{Purity}, \textit{$\Gamma$--spread}, \textit{$\Delta$--spread} \cite{custodio2011direct} and \textit{Hyper-volume} \cite{zitzler98}. \rev{For computing \textit{Purity} and \textit{Spreads}, a reference Pareto set is obtained combining the Pareto sets returned by all the solvers on a problem instance and removing the dominated solutions. For \textit{Hyper-volume} computation, on the other hand, the reference point $r_p$ is selected according to \cite{bras2020use} so that for all $i=1,\ldots,m$, $(r_p)_i = \max_{x\in {X}_\text{all}} f_i(x) + 0.01$, where $X_\text{all}$ is the union of all solutions sets returned by the considered algorithms.} Since \textit{Purity} and \textit{Hyper-volume} have higher values for better performance, they were pre-processed before being used for performance profiles; for \textit{Purity}, we considered the inverse of the obtained values. For the \textit{Hyper-volume}, we profiled the value $\rev{V_{\mathcal{R}}}-V_{\text{solver}}+\eta$, where $\rev{V_{\mathcal{R}}}$ is the hypervolume of the reference \rev{Pareto set} on a problem instance and $\eta=10^{-7}$ is for numerical reasons.  

\subsection{Preliminary Assessment of FD Methods Properties}
\label{subsec::prel_assess}

First, we evaluate the behavior of the Front Descent algorithm on two selected problems: MAN\_1 with $n=20$ and CEC09\_2 with $n=10$. In \cref{tab::MAN}, we analyze the progress of FD-SD on the two problems. 
We can note that, as anticipated in \cref{rem:finite_ref}, setting a tolerance $\sigma > 0$, after a problem-specific number of iterations FD-SD stopped carrying out refinement steps\rev{. Let us consider the case of MAN\_1. At iteration 60, no refinement step is carried out. However, there are iterations between the 60th and the 70th where some point is refined, as indicated in the row for iteration 70 by the value of ``(last)'', which shows that the last refinement occurred one iteration earlier. The value of ``(last)'' is not reset in subsequent rows, suggesting that iteration 69 is the final iteration where refinement is performed. A similar analysis can be applied to the CEC09\_2 problem, where the last refinement step appears to occur at iteration 148. From iteration 69 for MAN\_1 and iteration 148 for CEC09\_2} onward, each point generated by an exploration step is $\sigma$-Pareto-stationary.

\begin{table}
	\centering
	\footnotesize
	\caption{Details on the execution of FD-SD on MAN\_1 problem ($n=20$) and CEC09\_2 problem ($n=10$) with $\sigma_k=\sigma = 0.05$ for all $k$. For each considered iteration, we report: the number of initial solutions $|X^k|$ and the percentage of $\sigma$-Pareto-stationary points among them; the number $n_r$ of refinement steps performed during the iteration and the number of iterations passed since any refinement step was done; the number $n_e$ of exploration steps and the percentage of points produced in this phase that are $\sigma$-Pareto-stationary already; the number of points in the next iterate set $|X^{k + 1}|$.}
	\label{tab::MAN}
	\tiny
	\setlength\tabcolsep{2.5pt}
	\renewcommand{\arraystretch}{1.3}%
	\begin{tabular}{|c||c|c|c|c|}%
		\hline%
		\multicolumn{5}{|c|}{MAN\_1, $n = $ 20}\\
		\hline
		$k$&$|X^k|$ (\% stat.)&$n_r$ (last)&$n_e$ (\% stat.)&$|X^{k + 1}|$\\%
		\hline%
		\hline%
		1&10 (0)&8 (0)&4 (25)&12\\%
		\hline%
		10&40 (55)&16 (0)&5 (20)&43\\%
		\hline%
		20&102 (97)&3 (0)&12 (91)&114\\%
		\hline%
		30&259 (99)&2 (0)&27 (92)&280\\%
		\hline%
		40&488 (99)&2 (0)&48 (85)&510\\%
		\hline%
		50&778 (99)&2 (0)&77 (92)&828\\%
		\hline%
		60&1402 (100)&0 (1)&140 (100)&1479\\%
		\hline%
		70&2403 (100)&0 (1)&235 (100)&2551\\%
		\hline%
		80&4133 (100)&0 (11)&394 (100)&4361\\%
		\hline%
		90&6967 (100)&0 (21)&676 (100)&7350\\%
		\hline%
		100&11163 (100)&0 (31)&1054 (100)&11680\\%
		\hline%
	\end{tabular}%
	$\quad$
	\begin{tabular}{|c||c|c|c|c|}%
		\hline%
		\multicolumn{5}{|c|}{CEC09\_2, $n = $ 10}\\
		\hline
		$k$&$|X^k|$ (\% stat.)&$n_r$ (last)&$n_e$ (\% stat.)&$|X^{k + 1}|$\\%
		\hline%
		\hline%
		1&2 (50)&1 (0)&4 (25)&5\\%
		\hline%
		20&99 (98)&1 (0)&11 (90)&107\\%
		\hline%
		40&176 (98)&2 (0)&16 (87)&187\\%
		\hline%
		60&248 (99)&1 (0)&15 (93)&254\\%
		\hline%
		80&440 (99)&2 (0)&28 (92)&447\\%
		\hline%
		100&784 (100)&0 (5)&46 (100)&789\\%
		\hline%
		120&1150 (100)&0 (25)&79 (100)&1185\\%
		\hline%
		140&1325 (100)&0 (4)&94 (100)&1304\\%
		\hline%
		160&1437 (100)&0 (12)&101 (100)&1427\\%
		\hline%
		180&1753 (100)&0 (32)&133 (100)&1802\\%
		\hline%
		200&2015 (100)&0 (52)&140 (100)&2030\\%
		\hline%
	\end{tabular}
\end{table}

In such scenario, we can employ the hyper-volume based stopping condition proposed in \cref{rem:hypervol}. In \cref{fig::MAN_CEC}, we show the Pareto front reconstructions obtained by FD-SD on the two problems with different values for the \textit{hyper-volume} improvement threshold $\varepsilon_{hv}$ in \eqref{eq::hypervolume}. We conclude that properly setting $\varepsilon_{hv}$ allows us to control the quality of the generated Pareto front: in particular, employing smaller values for $\varepsilon_{hv}$ led us to more accurate Pareto front reconstructions, at the expense, reasonably, of an increased computational cost. This behavior is also well observed in the performance profiles (\cref{fig::PP_hypervolume}) on a wider set of bi-objective problems.

\begin{figure}
	\centering
	\subfloat{\includegraphics[width=0.45\textwidth]{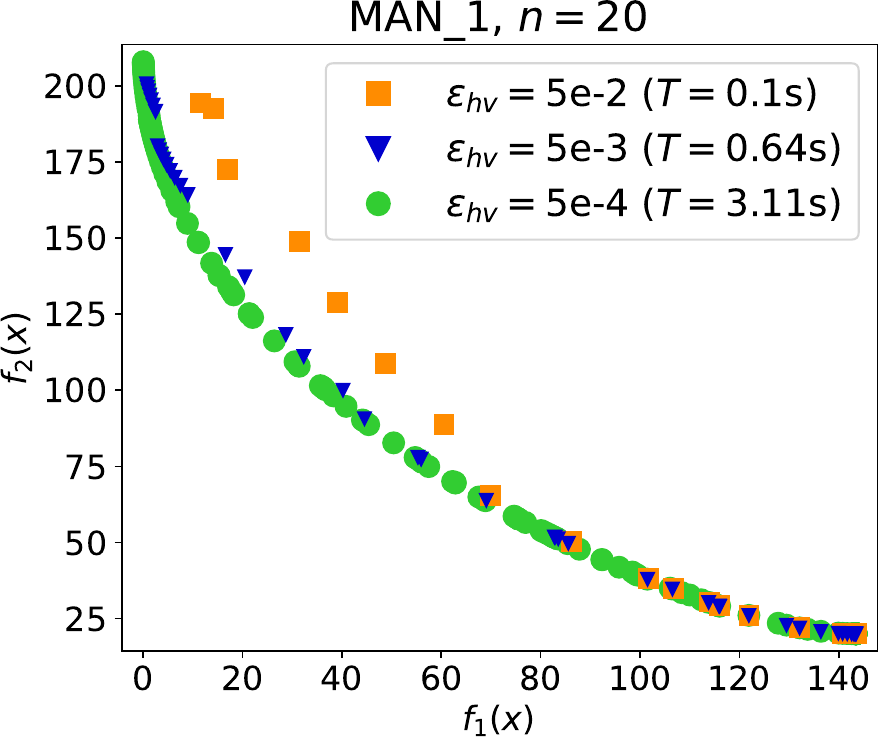}}
	\hfil
	\subfloat{\includegraphics[width=0.45\textwidth]{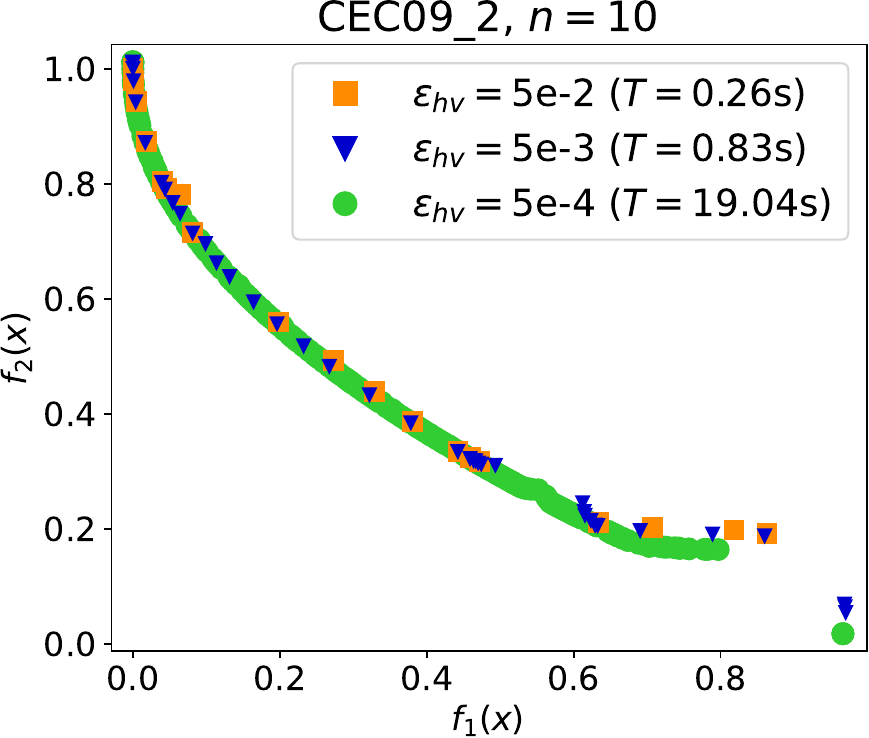}}
	\caption{Pareto front reconstructions by FD-SD on MAN\_1 ($n=20$) and CEC09\_2 ($n=10$) problems with different values for the \textit{Hyper-volume} improvement threshold $\varepsilon_{hv}$ (\cref{rem:hypervol}). For each execution, the runtime $T$ is reported.}
	\label{fig::MAN_CEC}
\end{figure}

\begin{figure}
	\centering
	\subfloat{\includegraphics[width=0.32\textwidth]{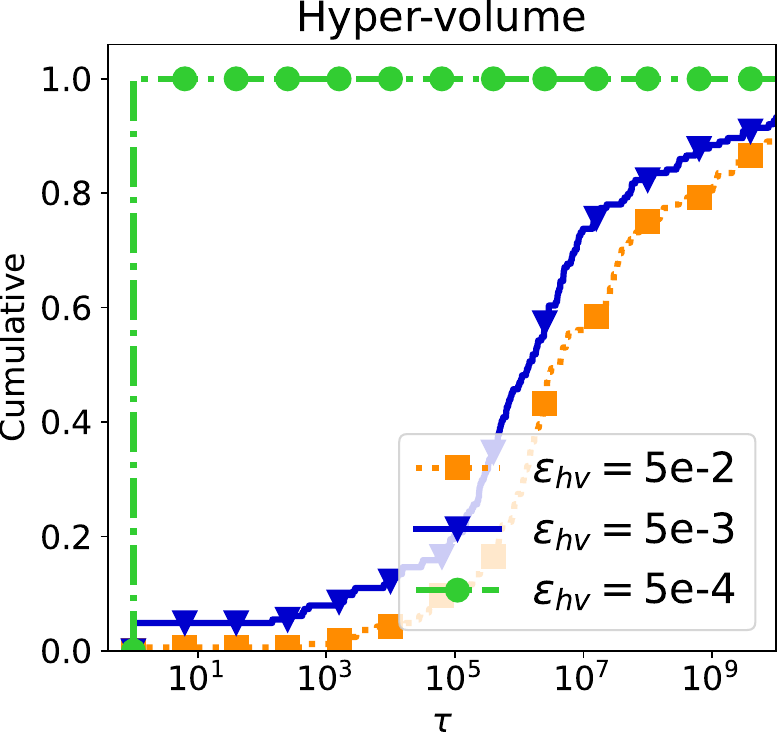}}
	\hfil
	\subfloat{\includegraphics[width=0.32\textwidth]{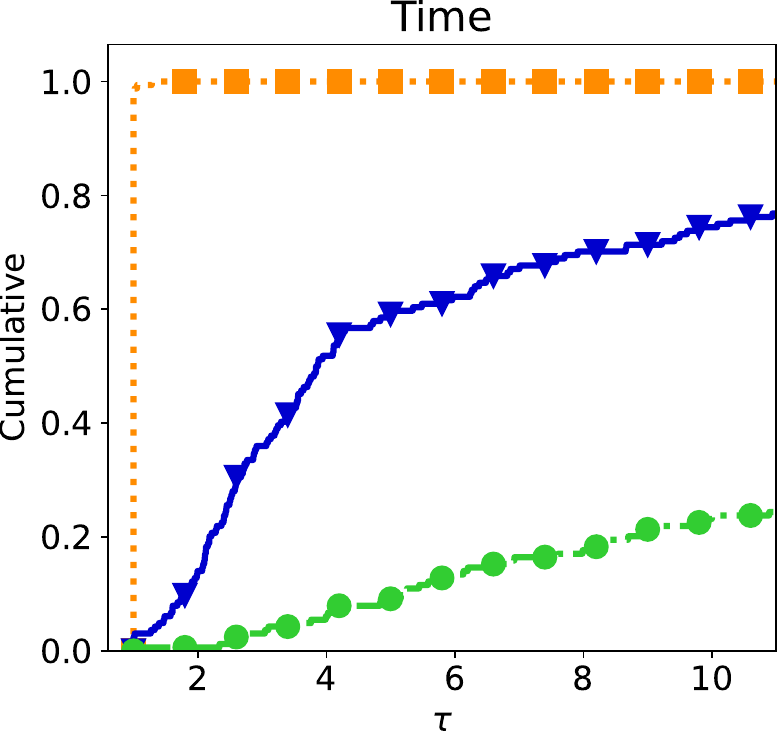}}
	\caption{Performance profiles w.r.t.\ \textit{Hyper-volume} and \textit{Time} for FD-SD on the bi-objective optimization problems listed in \cref{sec::computational_experiments}. The axes were set for a better visualization of the results.}
	\label{fig::PP_hypervolume}
\end{figure}

\subsection{Overall Comparison}

After assessing the properties typical of FD methods, we now compare all the FD variants on the set of bi-objective optimization problems listed in \cref{sec::computational_experiments}. Given the results in \cref{subsec::prel_assess}, here we considered $\varepsilon_{hv} = 5\times10^{-4}$ for all the algorithms. In \cref{fig::PP_FD}, we show the performance profiles w.r.t.\ \textit{Purity}, \textit{Time}, \textit{Hyper-volume} and \textit{Spread} metrics. As for the first metric, we observe that the employment of Newton-type and Barzilai-Borwein directions in the refinement phase actually helps to improve the quality of the Pareto front reconstruction w.r.t.\ the classical use of steepest descent direction. However, except for FD-BB, the use of additional information on the Newton-type directions led to an higher consumption of resources, as can be observed in the performance profiles w.r.t.\ \textit{Time}. Regarding the \textit{Hyper-volume} and $\Gamma$\textit{--spread} metrics, we have that FD-N was the best, with FD-BB having a similar performance, especially in robustness. As for $\Delta$\textit{--spread}, all the methodologies worked equally well.

\begin{figure}
	\centering
	\subfloat{\includegraphics[width=0.32\textwidth]{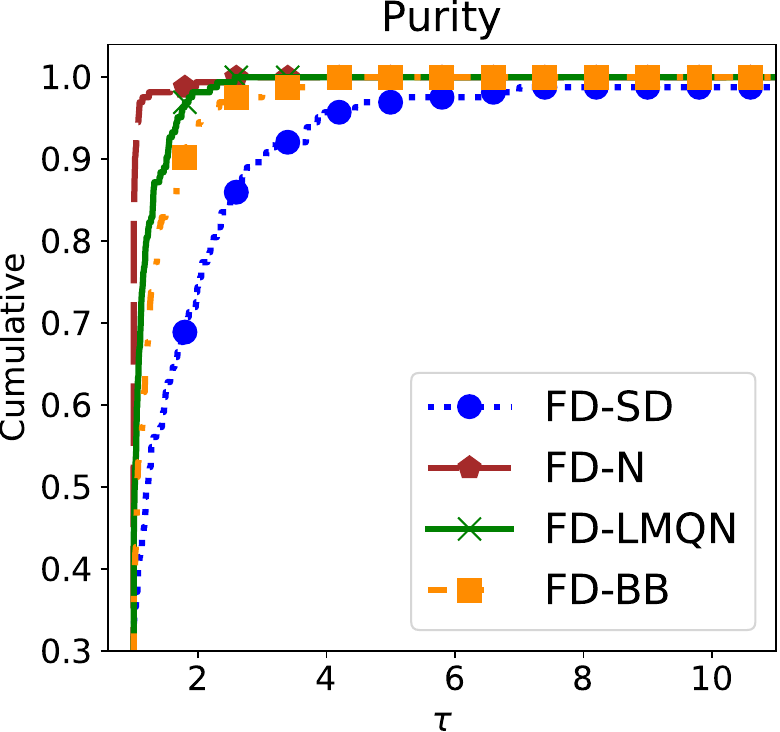}}
	\hfil
	\subfloat{\includegraphics[width=0.32\textwidth]{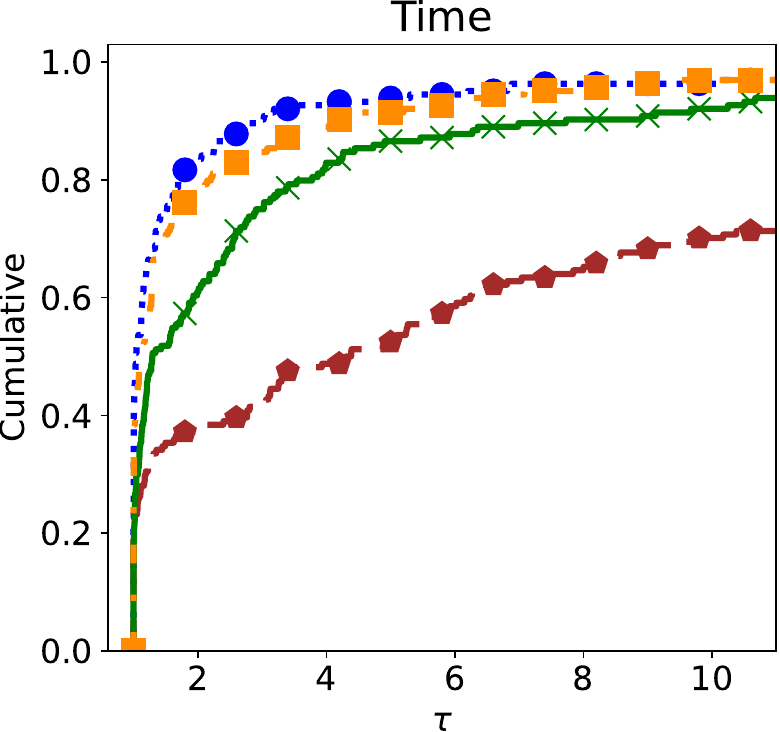}}
	\hfil
	\subfloat{\includegraphics[width=0.32\textwidth]{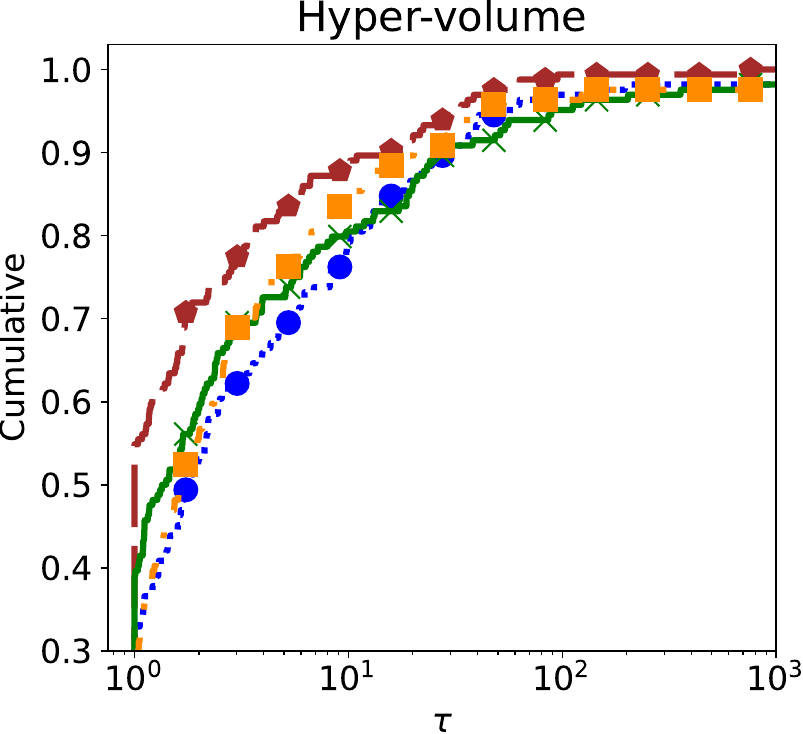}}
	\\
	\subfloat{\includegraphics[width=0.32\textwidth]{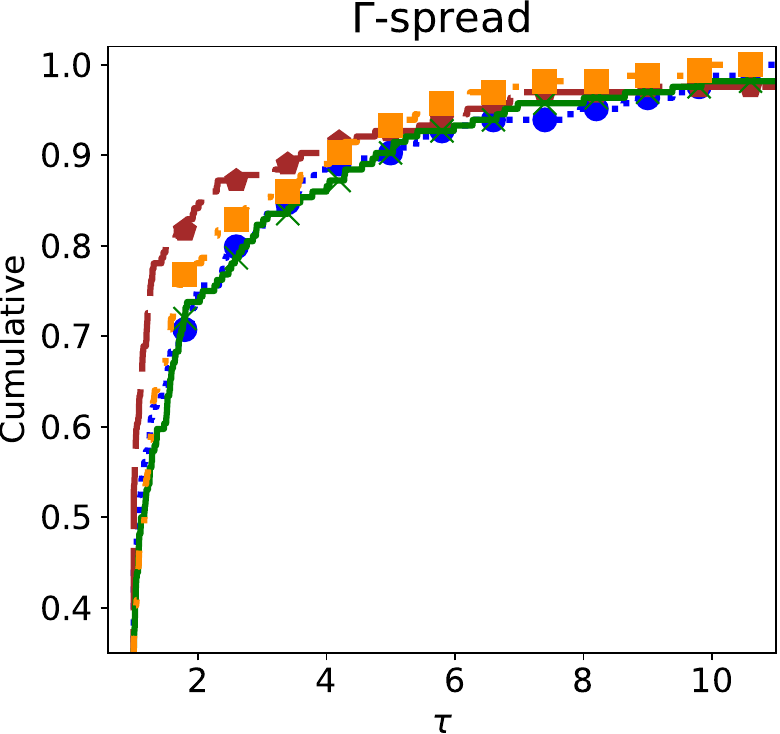}}
	\hfil
	\subfloat{\includegraphics[width=0.32\textwidth]{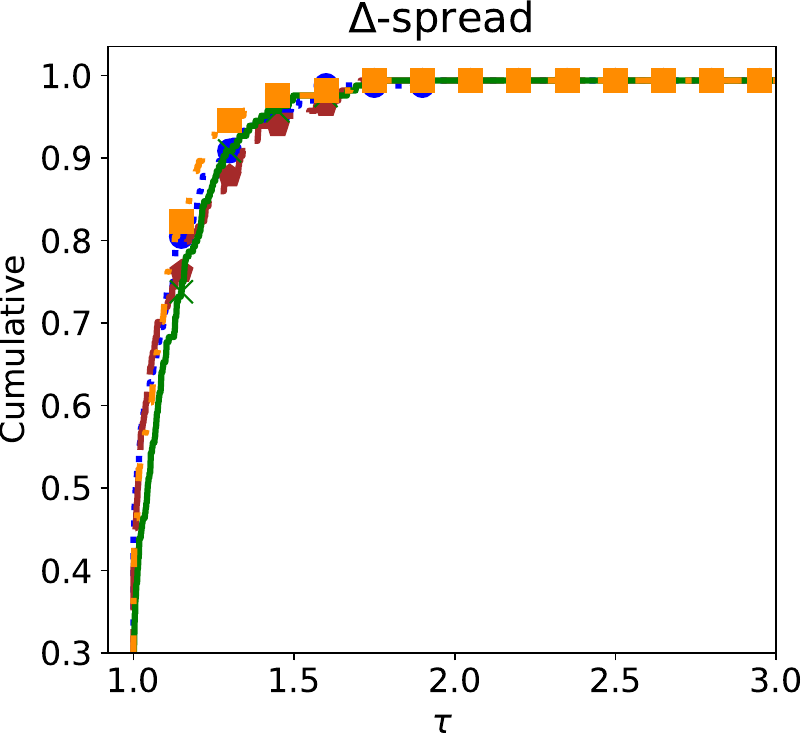}}
	\caption{Performance profiles w.r.t.\ \textit{Purity}, \textit{Time}, \textit{Hyper-volume} and \textit{Spread} metrics for FD-SD, FD-N, FD-LMQN, FD-BB on the bi-objective optimization problems listed in \cref{sec::computational_experiments}. The axes were set for a better visualization of the results.}
	\label{fig::PP_FD}
\end{figure}

For the last comparison, we then decided to use FD-BB, being \rev{competitive in terms of both effectiveness and efficiency}. In the comparisons with NSGA-II\rev{,} MOTR \rev{and DM-MADS} (\cref{fig::PP_state}), we \rev{initially} run all the algorithms with a time limit of 2 minutes on each problem, being the best stopping criterion to compare such structurally different approaches (of course, any additional stopping condition indicating that an algorithm cannot improve the set of solutions anymore was considered). FD-BB was the clear winner on both \textit{Purity} and \textit{Hyper-volume}. \rev{The very poor performance of DM-MADS is arguably explained by the high dimensionality of most of the tested problems: derivative-free methods like DM-MADS -- relying on a set of directions that positively span $\mathbb{R}^n$ to perform search steps at each iteration -- often tend to struggle in this kind of settings (also see, e.g., the discussion in \cite{lapucci2023memetic}).} Regarding the $\Gamma$\textit{--spread}, the proposed approach obtained a similar performance as MOTR on effectiveness and appeared to be the most robust algorithm; on $\Delta$\textit{--spread}, all the algorithms again performed equally well. \rev{Finally, we conducted the same experiments using time limits of 30 seconds and 5 minutes to investigate the sensitivity of our approach performance w.r.t.\ the time budget. The results, presented as performance profiles with respect to \textit{Hyper-volume} in \cref{fig::PP_state_HV_30s_5m}, show that, despite some negligible variations, our method consistently outperformed the others across all three time settings.}

\begin{figure}
	\centering
	\subfloat{\includegraphics[width=0.24\textwidth]{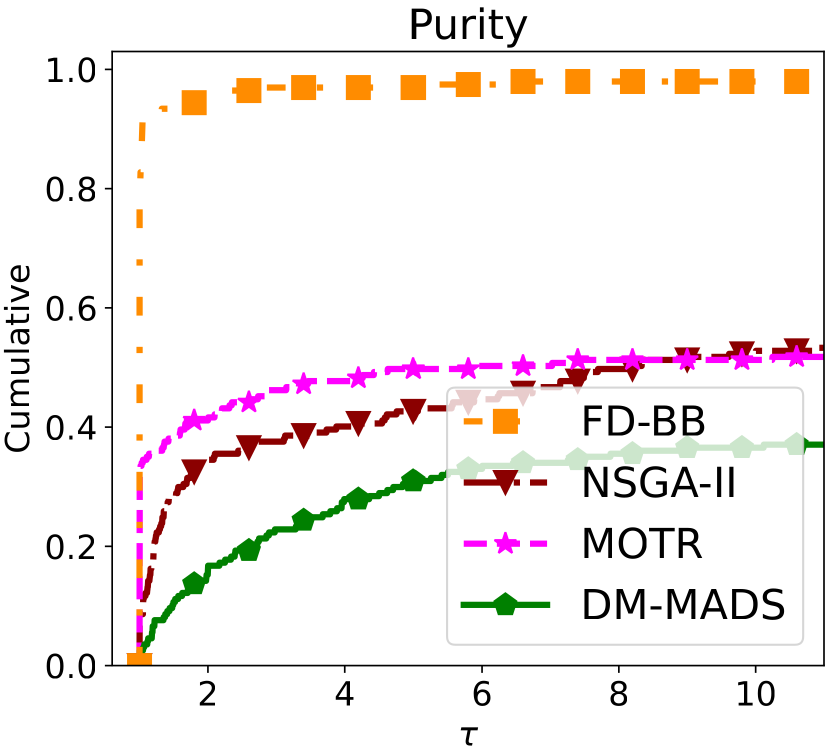}}
	\hfil
	\subfloat{\includegraphics[width=0.24\textwidth]{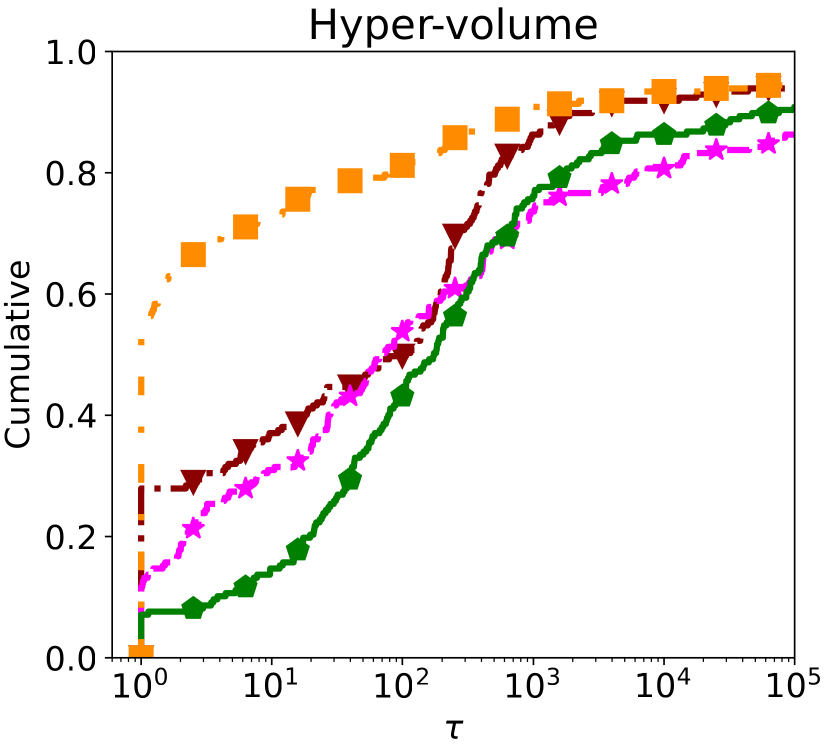}}
	\subfloat{\includegraphics[width=0.24\textwidth]{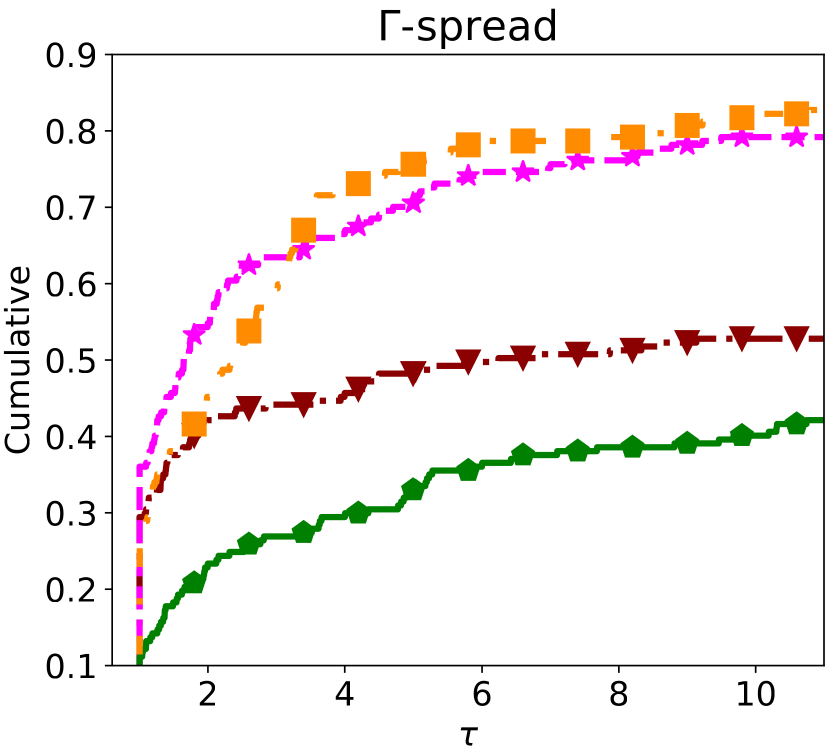}}
	\hfil
	\subfloat{\includegraphics[width=0.24\textwidth]{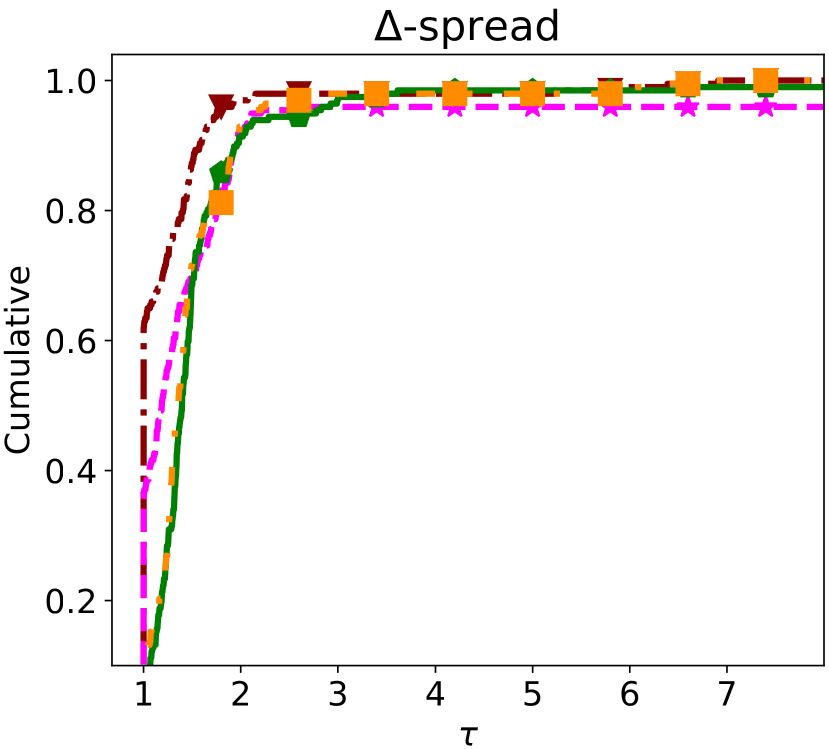}}
	\caption{Performance profiles w.r.t.\ \textit{Purity}, \textit{Hyper-volume} and \textit{Spread} metrics for FD-BB, NSGA-II\rev{,} MOTR \rev{and DM-MADS} on all the problems listed in \cref{sec::computational_experiments}. The axes were set for a better visualization of the results.}
	\label{fig::PP_state}
\end{figure}

\begin{figure}
	\centering
	\subfloat[\rev{Time limit of 30s}]{\includegraphics[width=0.24\textwidth]{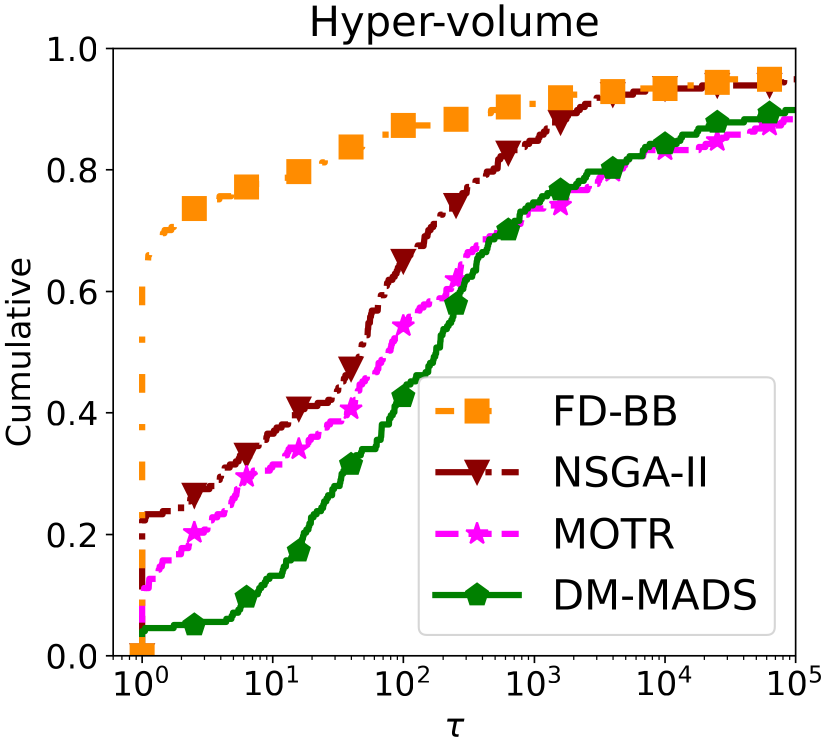}}
	\hfil
	\subfloat[\rev{Time limit of 2m}]{\includegraphics[width=0.24\textwidth]{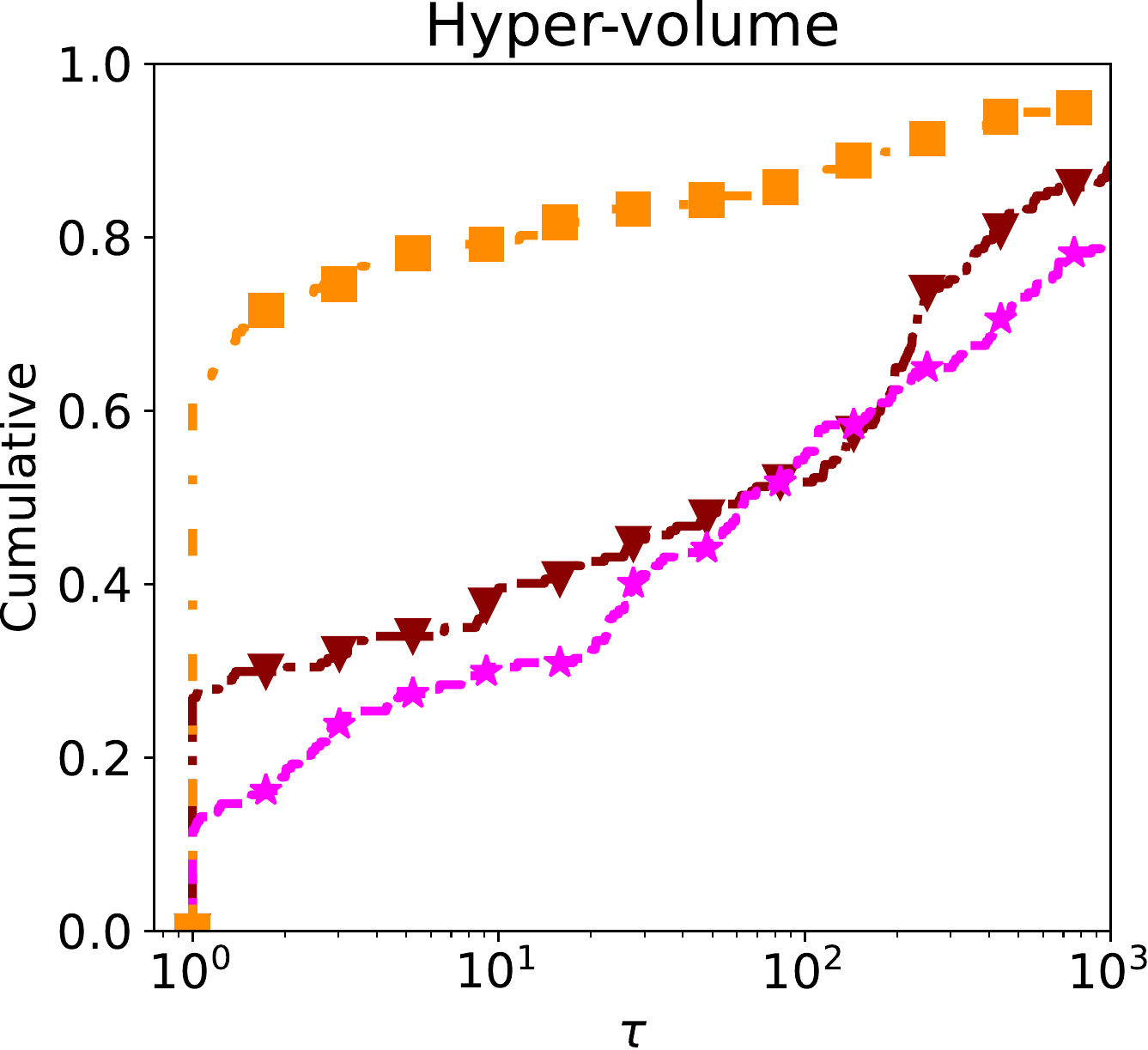}}
	\hfil
	\subfloat[\rev{Time limit of 5m}]{\includegraphics[width=0.24\textwidth]{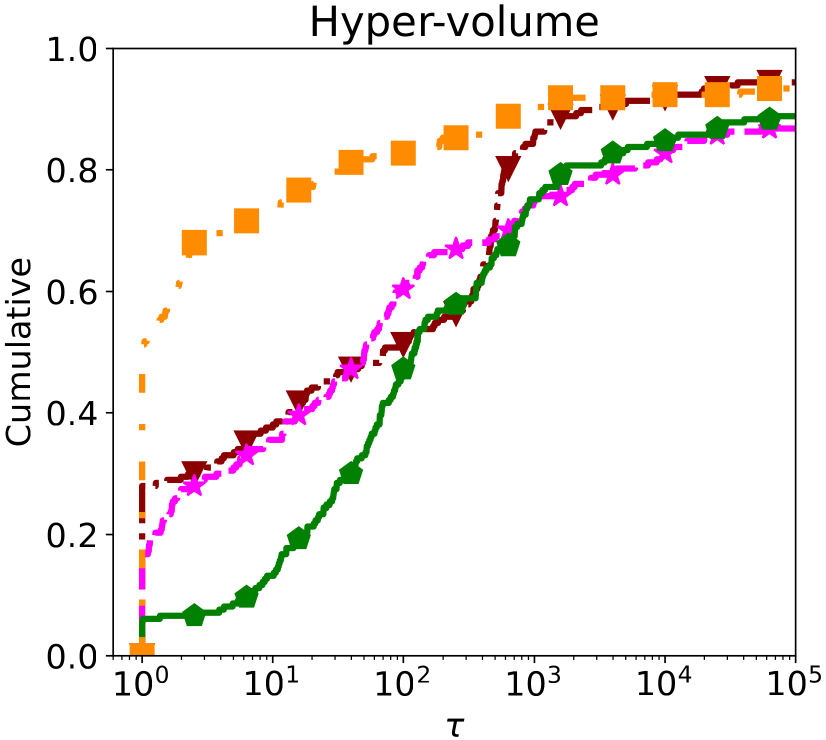}}
	\caption{\rev{Performance profiles w.r.t.\ \textit{Hyper-volume} for FD-BB, NSGA-II, MOTR and DM-MADS on all the problems listed in \cref{sec::computational_experiments}, run with different time limits. The axes were set for a better visualization of the results.}}
	\label{fig::PP_state_HV_30s_5m}
\end{figure}

\section{Conclusions}
\label{sec:conclusions}
In this paper, we introduced the class of Front Descent algorithms for Pareto front reconstruction in smooth multi-objective optimization.
For this general class of algorithms, we provided an insightful characterization of its mechanisms, a thorough theoretical analysis of convergence, with some innovative results concerning the sequence of iterate sets, and the experimental evidence of its superiority w.r.t.\ methods from the state-of-the-art. 

The theoretical analysis somehow sheds some light on the strong performance that the algorithm exhibits in terms of the purity metric. Interesting future research might focus on some formal investigation regarding the observed ability of the algorithm to effectively spread the Pareto front - or at least some of its parts in the nonconvex setting. 
Other research could of course be focused also on adaptations of the framework to the constrained setting.



\bibliographystyle{siamplain}

\end{document}